\documentclass[12pt]{amsart}

\newcommand{\cal}[1]{\mathcal{#1}}

\usepackage{john}
\usepackage{hyperref}
\usepackage{fullpage}

\numberwithin{theorem}{section}
\usepackage{color}
\usepackage{framed}

\makeatletter
\let\@@pmod\pmod
\DeclareRobustCommand{\pmod}{\@ifstar\@pmods\@@pmod}
\def\@pmods#1{\mkern4mu({\operator@font mod}\mkern 6mu#1)}
\makeatother

\newcommand{\NP}{\oper{NP}}
\newcommand{\BS}{\oper{BA}}

\newcommand{\ve}{\varepsilon}
\newcommand{\cP}{\mathcal P}
\newcommand{\bfloor}[1]{\left\lfloor#1\right\rfloor}

\DeclareMathOperator{\HZ}{HZ}
\DeclareMathOperator{\LZ}{LZ}

\usepackage[displaymath, mathlines]{lineno}

\usepackage{array}
\usepackage{multirow}

\usepackage[all,cmtip]{xy}
\usepackage[normalem]{ulem}
\usepackage{xcolor}
\newcommand\redsout{\bgroup\markoverwith{\textcolor{red}{\rule[0.5ex]{2pt}{1pt}}}\ULon}
\newcommand\bluesout{\bgroup\markoverwith{\textcolor{blue}{\rule[0.5ex]{2pt}{1pt}}}\ULon}

\date{\today}

\address{John Bergdall\\Department of Mathematics and Statistics \\ Boston University \\ 111 Cummington Mall \\ Boston, MA 02215\\USA}
\email{bergdall@math.bu.edu}
\urladdr{http://math.bu.edu/people/bergdall}

\address{Robert Pollack\\Department of Mathematics and Statistics \\ Boston University \\ 111 Cummington Mall \\ Boston, MA 02215\\USA}
\email{rpollack@math.bu.edu}
\urladdr{http://math.bu.edu/people/rpollack}


\title{Slopes of modular forms and the ghost conjecture\\(Unabridged version)}
\author{John Bergdall and Robert Pollack}

\begin{document}
\begin{abstract}
We formulate a conjecture on slopes of overconvergent $p$-adic cuspforms of any $p$-adic weight in the $\Gamma_0(N)$-regular case. This conjecture unifies a conjecture of Buzzard on classical slopes and more recent conjectures on slopes ``at the boundary of weight space''.
\end{abstract}
\maketitle
\setcounter{tocdepth}{1}
\tableofcontents

\section{Introduction and statement of the conjecture}

Let $p$ be a prime number, and let $N$ be a positive integer co-prime to $p$. The goal of this article is to investigate $U_p$-slopes:\ the $p$-adic valuations of the eigenvalues of the $U_p$-operator acting on spaces of (overconvergent $p$-adic) cuspforms of level $\Gamma_0(Np)$. Ultimately, we formulate a conjecture which unifies currently disparate predictions for the behavior of slopes at weights ``in the center'' and ``towards the boundary'' of $p$-adic weight space.

\subsection{Slopes of cuspforms}

The study of slopes of cuspforms began with extensive computer calculations of Gouv\^ea and Mazur in the 1990s \cite{GouveaMazur-FamiliesEigenforms}. Theoretical advancements of Coleman \cite{Coleman-pAdicBanachSpaces} led to a general theory of overconvergent $p$-adic cuspforms and eventually, with Mazur, to the construction of so-called eigencurves \cite{ColemanMazur-Eigencurve}. To better understand the geometry of the newly constructed eigencurves, Buzzard and his co-authors returned to explicit investigations on slopes in a series of papers \cite{Buzzard-SlopeQuestions,BuzzardCalegari-2adicSlopes,BuzzardCalegari-GouveaMazur,BuzzardKilford-2adc}. 

In \cite{Buzzard-SlopeQuestions}, Buzzard produced a combinatorial algorithm (``Buzzard's algorithm'') that for fixed $p$ and $N$ takes as input $k$ and outputs $\dim S_k(\Gamma_0(N))$-many integers.  He also defined the notion of a prime $p$ being $\Gamma_0(N)$-regular and conjectured that his algorithm was computing slopes in the regular cases.\footnote{Buzzard's algorithm only outputs integers, so Buzzard's conjecture implies that $U_p$-slopes are always integral in $\Gamma_0(N)$-regular cases.} 
\begin{definition}[{\cite[Definition 1.2]{Buzzard-SlopeQuestions}}]\label{definition:Gamma0(N)-regular}
An odd prime $p$ is $\Gamma_0(N)$-regular if the Hecke operator $T_p$ acts on $S_k(\Gamma_0(N))$ with $p$-adic unit eigenvalues for $2 \leq k \leq p+1$.
\end{definition}
See Definition \ref{defn:2regular} for $p=2$, but we note now that $p=2$ is $\SL_2(\Z)$-regular. The first prime $p$ which is not $\SL_2(\Z)$-regular is $p=59$.

Buzzard's algorithm is concerned with spaces of cuspforms without character, where the slopes vary in a fairly complicated way. By contrast, a theorem of Buzzard and Kilford  \cite{BuzzardKilford-2adc} implies that if $j \geq 3$ and $\chi$ is a primitive Dirichlet character of conductor $2^j$ then the $U_2$-slopes in $S_k(\Gamma_1(2^j),\chi)$ are the neatly ordered numbers $ 2^{3-j} \cdot \left(1,2,3,\dotsc,k-2\right)$. See also analogous theorems of Roe \cite{Roe-Slopes}, Kilford \cite{Kilford-5Slopes} and Kilford--McMurdy \cite{KilfordMcMurday-7adicslopes}.

In \cite{LiuXiaoWan-IntegralEigencurves}, Liu, Wan and Xiao gave  a conjectural, but general, framework in which to view the Buzzard--Kilford calculation (see \cite{WanXiaoZhang-Slopes} also). Namely, those authors have conjectured that the slopes of the $U_p$-operator acting on spaces of overconvergent $p$-adic cuspforms at $p$-adic weights ``near the boundary of weight space'' are finite unions of arithmetic progressions whose initial terms are the slopes in explicit classical weight two spaces. They also verified their conjecture for overconvergent forms on definite quaternion algebras.

The beautiful description of the slopes at the boundary of weight space is actually a consequence (see \cite{BergdallPollack-FredholmSlopes,{LiuXiaoWan-IntegralEigencurves}}) of a conjecture, widely attributed to Coleman, called ``the spectral halo'':\ after deleting a closed subdisc of $p$-adic weight space, the Coleman--Mazur eigencurve becomes an infinite disjoint union of finite flat covers over the remaining portion of weight space. Furthermore, families of eigenforms over outer annuli of weight space should be interpreted as $p$-adic families passing through overconvergent $p$-adic eigenforms in characteristic $p$ (see \cite{AndreattaIovitaPilloni-Halo, JohanssonNewton-Extended}). The existence of a spectral halo should not depend on regularity.

In summary, for a space either of the form $S_k(\Gamma_0(Np))$ or $S_k(\Gamma_0(N)\intersect \Gamma_1(p^r),\chi)$, the slopes are conjectured to be determined by a finite computation in small weights together with an algorithm:\ Buzzard's algorithm in the first case and ``generate an arithmetic progression'' in the second. 

In this article, we present a unifying conjecture that predicts the slopes of overconvergent $p$-adic eigenforms over all of $p$-adic weight space simultaneously. The shape of our conjecture is the following:\ we write down a power series over $\Z_p$ in two variables, one of which is the weight variable. We then conjecture, in the $\Gamma_0(N)$-regular case, that the Newton polygon of the specialization of our series to any given weight has the same set of slopes as the $U_p$-operator acting on the corresponding space of overconvergent $p$-adic cuspforms.

\subsection{Fredholm series}
Our approach begins with overconvergent $p$-adic modular forms. Write $\cal  W$ for the {\em even} $p$-adic weight space:\ the space of continuous characters $\kappa: \Z_p^\times \rightarrow \C_p^\times$ with $\kappa(-1)=1$. For each $\kappa \in \cal W$ we write $S_\kappa^{\dagger}(\Gamma_0(Np))$ for the space of weight $\kappa$ overconvergent $p$-adic cuspforms of level $\Gamma_0(Np)$.

An integer $k$ gives rise to a $p$-adic weight $z \mapsto z^k$, and  the finite-dimensional space $S_k(\Gamma_0(Np))$ sits as a $U_p$-stable subspace of $S_k^{\dagger}(\Gamma_0(Np))$. In \cite{Coleman-ClassicalandOverconvergent}, Coleman proved that the $U_p$-slopes in $S_k(\Gamma_0(Np))$ are almost exactly those $U_p$-slopes in $S_k^{\dagger}(\Gamma_0(Np))$ which are at most $k-1$. Thus, one could determine the classical slopes by attempting the seemingly more difficult task of determining the overconvergent slopes.

Denote by
\begin{equation*}
P_\kappa(t) = \det\left(1 - t\restrict{U_p}{S_{\kappa}^{\dagger}(\Gamma_0(Np))}\right) = 1 + \sum_{i \geq 1} a_i(\kappa)t^i \in \Q_p[[t]]
\end{equation*}
the Fredholm series for the $U_p$-operator in weight $\kappa$. The series $P_\kappa$ is entire in the variable $t$ and the $U_p$-slopes in weight $\kappa$ are the slopes of the segments of the Newton polygon of $P_\kappa$.  Coleman's proved (see \cite[Appendix I]{Coleman-pAdicBanachSpaces}) that $\kappa \mapsto a_i(\kappa)$ is defined by a power series with $\Z_p$-coefficients.

To be precise, we write $\cal W = \bigunion_{\ve} \cal W_{\ve}$ where the (disjoint) union runs over even characters $\varepsilon:(\Z/2p\Z)^\times \rightarrow \C_p^\times$, and $\kappa \in \cal W$ is in $\cal W_{\ve}$ if and only if the restriction of $\kappa$ to the torsion subgroup in $\Z_p^\times$ is given by $\ve$. We fix a topological generator $\gamma$ for the procyclic group $1+2p\Z_p$. Each $\cal W_{\varepsilon}$ is then an open $p$-adic unit disc with coordinate $w_{\kappa} = \kappa(\gamma)-1$.

The meaning of Coleman's second result can now be clarified:\ for each $\varepsilon$ there exists a two variable series
\begin{equation*}
P^{(\varepsilon)}(w,t) = 1 + \sum_{i=1}^\infty a_i^{(\varepsilon)}(w)t^i \in \Z_p[[w,t]]
\end{equation*}
such that if $\kappa \in \cal W_{\varepsilon}$ then $P_{\kappa}(t) = P^{(\varepsilon)}(w_{\kappa},t)$. In particular, the slopes of overconvergent $p$-adic cuspforms are encoded in the Newton polygons of the evaluations of the $P^{(\varepsilon)}$ at $p$-adic weights.

\subsection{The ghost conjecture}\label{subsec:ghost-conjecture}
Our approach to predicting slopes is to create a faithful, explicit, model $G^{(\varepsilon)}$ for each Fredholm series $P^{(\varepsilon)}$. We begin by writing $G^{(\varepsilon)}(w,t) = 1 + \sum g_i^{(\varepsilon)}(w)t^i$ for coefficients $g_i^{(\varepsilon)}(w)$ which we shortly determine. If decorations are not needed, we refer to $g(w)$ as one of these coefficients. Each coefficient will be non-zero and not divisible by $p$.\footnote{In \cite{BergdallPollack-FredholmSlopes}, the authors showed that if $N=1$ then the coefficients $a_i^{(\varepsilon)}(w)$ are not divisible by $p$. For $N >1$ this is not true, but we don't believe this divisibility plays a crucial role for predicting slopes.} In particular, $w_{\kappa} \mapsto v_p(g(w_{\kappa}))$ will depend only on the relative position of $w$ to the finitely many roots of $g(w)$ in the open disc $v_p(w) > 0$.

To motivate our specification of the zeros of $g_i^{(\varepsilon)}(w)$, we make two observations:
\begin{enumerate}
\item If $g_i^{(\varepsilon)}(w_\kappa) = 0$ then the $i$-th and $(i+1)$-st slope of the Newton polygon of $G^{(\varepsilon)}(w_\kappa,t)$ are the same.
\end{enumerate}
So, one can ask:\ what are the slopes that appear with multiplicity in spaces of overconvergent $p$-adic cuspforms? The second observation is:
\begin{enumerate}
\setcounter{enumi}{1}
\item If $k\geq 2$ is an even integer then the slope ${k-2\over 2}$ is often repeated in $S_k^{\dagger}(\Gamma_0(Np))$.
\end{enumerate}
In fact,  any eigenform in $S_k(\Gamma_0(Np))$ which is new at $p$ has slope ${k-2\over 2}$. So, in order to model the slopes of $U_p$ it might be reasonable to insist that $g_i^{(\varepsilon)}(w)$ has a zero exactly at $w = w_{k}$ with $k \in \cal W_{\varepsilon}$ where the $i$-th and $(i+1)$-st slope of $U_p$ acting on $S_k(\Gamma_0(Np))$ are both ${k-2\over 2}$. This leads us to seek $g_i^{(\ve)}$ such that:
\begin{equation}\label{eqn:mi_positive}
g_i^{(\varepsilon)}(w_{k}) = 0 \iff \dim S_k(\Gamma_0(N)) < i < \dim S_k(\Gamma_0(N)) + \dim S_k(\Gamma_0(Np))^{p-\new}
\end{equation}
for $k \in \cal W_{\ve}$. Such a $g_i^{(\varepsilon)}$ exists because for fixed $i$, the right-hand side of \eqref{eqn:mi_positive} holds for at most finitely many $k$.

We now need to specify the multiplicities of the zeros $w_k$.\footnote{The na\"ive idea of having all zeros of $g_i^{(\ve)}$ be simple would not work because the ghost series defined below would not be an entire series (compare with Proposition \ref{proposition:ghost-entire}).}  An integer $k \in \cal W_{\varepsilon}$ is a zero for $g_i^{(\varepsilon)}(w)$ for some range of consecutive integers $i = a,a+1,\dots,b$ for which the right-hand side of \eqref{eqn:mi_positive} holds.  Roughly, we set the order of vanishing of $g_a^{(\varepsilon)}(w)$ and $g_b^{(\varepsilon)}(w)$ at $w = w_{k}$  to be 1; for $g_{a+1}^{(\varepsilon)}(w)$ and $g_{b-1}^{(\varepsilon)}(w)$ to be 2; and so on.  More formally, define the sequence $s(\ell)$ by
\begin{equation*}
s_i(\ell) = \begin{cases}
i & \text{if $1\leq i \leq \floor{\ell / 2}$}\\
\ell+1 - i & \text{if $\floor{\ell / 2} < i \leq \ell$},
\end{cases}
\end{equation*}
and $s(\ell)$ is the empty sequence if $\ell\leq 0$. 
For $d \geq 0$ we write $s(\ell,d)$ for the infinite sequence
\begin{equation*}
s(\ell,d) = (\underlabel{d\text{ times}}{0,\dotsc,0},s_1(\ell),s_2(\ell),\dotsc,s_\ell(\ell),0,\dotsc).
\end{equation*}
If $k$ is an integer then set $d_k := \dim S_k(\Gamma_0(N))$ and $d_k^{\new} := \dim S_k(\Gamma_0(Np))^{p-\new}$. We then define $m(k) = s(d_k^{\new}-1,d_k)$, and set
\begin{equation*}
g_i^{(\varepsilon)}(w) := \prod_{k \in \cal W_{\varepsilon}} (w-w_{k})^{m_i(k)} \in \Z_p[w] \subset \Z_p[[w]]
\end{equation*}
which we note is a finite product.

\begin{definition}
The $p$-adic ghost series of tame level $\Gamma_0(N)$ on the component $\cal W_{\varepsilon}$ is
\begin{equation*}
G^{(\varepsilon)}(w,t) := 1 + \sum_{i=1}^\infty g_i^{(\varepsilon)}(w)t^i \in \Z_p[[w,t]].
\end{equation*}
\end{definition}
The naming choice and the motivation for the multiplicities defined in the next paragraph are discussed in Appendix \ref{app:explicit-2adic}. We check in Proposition \ref{proposition:ghost-entire} that $G^{(\varepsilon)}$ is entire as a power series in the variable $t$ over $\Z_p[[w]]$.  In particular, for each $p$-adic weight $\kappa$ we get an entire series $G_{\kappa} \in \C_p[[t]]$.  In what follows, we write $\NP(-)$ for ``Newton polygon''.

\begin{conjecture}[The ghost conjecture]\label{conj:intro-ghost}
If $p$ is an odd $\Gamma_0(N)$-regular prime or $p=2$ and $N = 1$, then $\NP(G_\kappa) = \NP(P_\kappa)$ for each $\kappa \in \cal W$.
\end{conjecture}

We check below that the hypotheses on $p$ is necessary (Theorem \ref{thm:ghost-true-regular}).  In Section \ref{subsec:modification}, we formulate a conjecture when $p=2$ is $\Gamma_0(N)$-regular with a modified ghost series.

\pagebreak

\subsection{Evidence for the ghost conjecture}

\subsubsection{Buzzard's conjecture versus the ghost conjecture}
Buzzard's algorithm exploits many known and conjectured properties of slopes, such as their internal symmetries in classical subspaces, their (conjectural) local constancy in large families, and their interaction with Coleman's $\theta$-operator, to recursively predict classical $U_p$-slopes. The ghost conjecture on the other hand, simply motivated by the properties of slopes of $p$-newforms, predicts all overconvergent $U_p$-slopes and one obtains classical slopes by keeping the first $d_k$-many. These two approaches are completely different and, yet, they appear to agree. We view such agreement as compelling evidence for both conjectures. 

If $G(t) \in 1 + t\C_p[[t]]$ is a power series and $d \geq 1$, then write $G^{\leq d}$ for the truncation of $G$ in degree $\leq d$. Write $\BS(k)$ for the output of Buzzard's algorithm on input $k$. 

\begin{fact}\label{fact:buzzard-agreement}
If either
\begin{enumerate}
\item $N=1$ and $p\leq 4099$ and $2\leq k\leq 2050$, or
\item $2 \leq N \leq 42$,  $3 \leq p \leq 199$ and $2 \leq k \leq 400$,
\end{enumerate}
then the multiset of slopes of $\NP(G_k^{\leq d_k})$ is equal to $\BS(k)$.
\end{fact}
We note that Buzzard made an extensive numerical verification of his conjecture which included all weights $k\leq 2048$ for $p=2$ and $N=1$.

The careful reader will note a striking omission in the statement of Fact \ref{fact:buzzard-agreement}:\ the agreement between the ghost slopes and the output of Buzzard's algorithm does not seem to be limited to $\Gamma_0(N)$-regular cases. Namely, neither the construction of the ghost series nor Buzzard's algorithm requires any {\em a priori} regularity hypotheses and the tests we ran to check Fact \ref{fact:buzzard-agreement} were not limited to regular cases. It seems possible that someone with enough patience could even prove, without any hypothesis on $p$ and $N$, that the output of Buzzard's algorithm agrees with the classical ghost slopes. Although neither conjecture is predicting $U_p$-slopes in the irregular case, the numbers they both output could be thought of as representing the $U_p$-slopes that ``would have occurred" if not for the existence of a non-ordinary form of low weight.

\subsubsection{Comparisons with known theorems on slopes}
There are a number of cases where the slopes of $\NP(P_{\kappa})$ have been determined. In such cases that we know of, we independently verify that the ghost series determines the same list of slopes. 
\begin{theorem}[Theorem \ref{theorem:agree-BK-p=2}, Corollary \ref{corollary:bc-k=0}, Theorem \ref{theorem:consisten-boundary}]\label{theorem:intro-actual-truth}
$\NP(G_{\kappa}) = \NP(P_{\kappa})$ in the following cases:
\begin{enumerate}
\item $p=2$, $N=1$, $\kappa=0$,
\item $p=2$, $N=1$, $v_2(w_{\kappa}) < 3$,
\item $p=3$, $N=1$, $v_3(w_{\kappa}) < 1$,
\item $p=5$, $N=1$, $\kappa$ of the form $z^k\chi$ with $\chi$ conductor $25$ and
\item $p=7$, $N=1$, $\kappa \in \mathcal W_0 \cup \mathcal W_2$ of the form $z^k\chi$ with $\chi$ conductor $49$.
\end{enumerate}
\end{theorem}
The determination of the $U_p$-slopes in these cases are due to, in order, Buzzard and Calegari \cite{BuzzardCalegari-2adicSlopes}, Buzzard and Kilford \cite{BuzzardKilford-2adc}, Roe \cite{Roe-Slopes}, Kilford \cite{Kilford-5Slopes} and Kilford and McMurdy \cite{KilfordMcMurday-7adicslopes}.

We also check the ghost conjecture is consistent with a conjecture of Buzzard and Calegari in \cite{BuzzardCalegari-2adicSlopes} on $2$-adic, tame level one, slopes at negative integers (Theorem \ref{theorem:agree-with-buzzard-calegari}) and we derive formulas for the slopes of $\NP(G_0)$ when $p=3,5$ and $N=1$ which agree with formulas found in Loeffler's paper \cite{Loeffer-SpectralExpansions} (Proposition \ref{prop:loeffler-sequences}).

\subsubsection{The ghost spectral halo}\label{subsec:ghost-spectral-intro}
Coleman's spectral halo, mentioned above, is concerned with $p$-adic weights quite far away from the integers. Specifically, let us refer to the spectral halo as the conjecture: 
\begin{conjecture}[The spectral halo conjecture]\label{conj:spectral-halo}
There exists a $v > 0$ such that ${1\over v_p(w_\kappa)}\NP(P_\kappa)$ is independent of $\kappa \in \cal{W}_{\varepsilon}$ if $0 < v_p(w_\kappa) < v$.
\end{conjecture}
On $\cal W_{\ve}$, the constant value of ${1\over v_p(w_\kappa)}\NP(P_\kappa)$ is then beautifully realized as the $w$-adic Newton polygon $\NP(\bar P)$ where $\bar P$ is the mod $p$ reduction of the $P^{(\varepsilon)}$.

The ghost series trivially satisfies this halo-like behavior. Indeed, the zeros of each coefficient $g(w)$ lie in the region $v_p(w_\kappa) \geq 1$ (or $v_2(w_\kappa) \geq 3$ if $p=2$). Thus, over the complement of those regions, we have $v_p(g(w_{\kappa})) = \lambda(g)v_p(w_{\kappa})$ where $\lambda(g) = \deg g$. This proves:
\begin{theorem}[The ghost spectral halo]
\label{thm:gsh}
The function $\kappa \mapsto {1\over v_p(w_\kappa)}\NP(G_\kappa)$ is independent of $\kappa \in \cal W_\ve$ if $0 < v_p(w_\kappa) < 1$ (and $0 < v_2(w_{\kappa}) < 3$ if $p=2$), and the constant value is equal to $\NP(\bar{G}^{(\varepsilon)})$.
\end{theorem}

Along with the spectral halo conjecture, one also predicts that the slopes of $\NP(\bar P)$ are a finite union of arithmetic progressions for $v_p(w_{\kappa})$ small (see \cite[Conjecture 1.2(3)]{LiuXiaoWan-IntegralEigencurves}).  We prove this directly for the ghost series, up to finite error. Write $\mu_0(N)$ for the index of $\Gamma_0(N)$ inside $\SL_2(\Z)$.

\begin{theorem}[{Corollary \ref{cor:progs}, Theorem \ref{theorem:agree-BK-p=2}}]\label{theorem:intro-aps-boundary}
If $p$ is odd then the slopes of $\NP(\bar G)$ are a finite union of ${p(p-1)(p+1)\mu_0(N)\over 24}$-many arithmetic progressions with a common difference ${(p-1)^2\over 2}$, except for finitely many possible exceptional slopes.

If $p=2$ and $N=1$ then $\NP(\bar G)$ has slopes $\set{1,2,3,\dotsc}$.
\end{theorem}
We note that Theorem \ref{thm:gsh} and \cite{BergdallPollack-FredholmSlopes} imply that if the ghost conjecture is true then the exceptional slopes do not appear. More specifically, if the ghost conjecture is true then Theorem \ref{thm:gsh} implies the spectral halo exists on $0 < v_p(w_{\kappa}) < 1$, and if that is true then \cite[Theorem 3.10]{BergdallPollack-FredholmSlopes} proves that the slopes in Theorem \ref{theorem:intro-aps-boundary} are a finite union of arithmetic progressions without exceptions. Moreover, as evidence for the ghost conjecture, one can independently verify that the number of progressions predicted by  \cite[Theorem 3.10]{BergdallPollack-FredholmSlopes} is exactly the same number written in Theorem \ref{theorem:intro-aps-boundary}.\footnote{If Conjecture \ref{conj:spectral-halo} is true with $v = 1$ then \cite[Theorem 3.10]{BergdallPollack-FredholmSlopes} predicts the number of progressions with common difference ${(p-1)^2\over 2}$ is given by
\begin{equation*}
{(p-1)c_0(N)\over 2} + \sum_{j=0}^{{p-3\over 2}} \dim S_2(\Gamma_0(N)\cap \Gamma_1(p^2), \chi\omega^{-2j})
\end{equation*}
where $\chi$ is an even primitive character modulo $p^2$, $\omega$ is the Teichm\"uller character and $c_0(N)$ is the number of cusps on $X_0(N)$. One can check that this is exactly ${p(p-1)(p+1)\mu_0(N)\over 24}$ (using \cite[Th\'eor\`eme 1]{CohenOesterle-Dimensions} for example).
} 

In addition to the ghost spectral halo, we've also discovered interesting arithmetic properties of slopes over {\em other} regions of $p$-adic weight space. See Section \ref{subsec:intro-global-halo} below (specifically Theorem \ref{theorem:global-halo}, which is a vast generalization of Theorem \ref{theorem:intro-aps-boundary}).

\subsection{Distribution of ghost slopes}
In Theorem \ref{thm:asymptotic} below we prove an asymptotic formula for the $i$-th slope of $\NP(G_k)$ when $k\geq 2$ is an even integer. Here we highlight two corollaries related to conjectures of Buzzard--Gouv\^ea and Gouv\^ea on the distribution of classical slopes. We write $s_1(k) \leq s_2(k) \leq \dotsb$ for the slopes of $\NP(G_k)$.

On \cite[Page 8]{Gouvea-WhereSlopesAre}, Gouv\^ea asks  if $v_p(a_p) \leq {k-1\over p+1}$ with probability one as $k\goto \infty$ where $a_p$ ranges over eigenvalues for $T_p$ acting on $S_k(\Gamma_0(N))$. Buzzard asks in \cite[Question 4.9]{Buzzard-SlopeQuestions} if the bound is {\em always} true when $p$ is $\Gamma_0(N)$-regular.\footnote{Gouv\^ea also asks whether or not $v_p(a_p) \leq {k-1\over p+1}$ for all $k$ once it is true for $k \leq p+1$, which is a slightly stronger question (\cite[Page 9]{Gouvea-WhereSlopesAre}).}  We prove that the ghost slopes satisfy an asymptotic version of the Buzzard--Gouv\^ea bound.

\begin{theorem}[Corollary \ref{cor:highest}]\label{theorem:buzzard-gouvea-intro}
For $k\geq 2$ even, 
\begin{equation*}
s_{d_k}(k) = {k \over p+1} + O(\log k).
\end{equation*}
\end{theorem}
We believe that in fact $s_{d_k}(k) \leq {k-1\over p+1}$ always, but we did not pursue this except if $p=2$. We will not include details here.

In \cite{Gouvea-WhereSlopesAre}, Gouv\^ea also considered, for a fixed $k$, the set 
\begin{equation*}
\mathbf x_k := \set{ {h\over k-1} \st h \text{ is a slope of $T_p$ acting on $S_k(\Gamma_0(N))$}} \subset [0,1].
\end{equation*}
He then conjectured that the sets $\mathbf x_k$ become equidistributed on $[0,{1\over p+1}]$ as $k\goto \infty$. We establish an analogous property for the ghost slopes. Write $d_{k,p} = \dim S_k(\Gamma_0(Np))$.

\begin{theorem}[{Corollary \ref{corollary:gouvea-distribution}}]
\label{thm:gouvea-dist}
As $k\goto \infty$, the sets
\begin{equation*}
\set{ {s_i(k)\over k-1} \st 1\leq i \leq d_{k,p}}
\end{equation*}
become equidistributed with respect to the unique probability measure on $[0,{1\over p+1}] \union \set{1\over 2} \union [{p\over p+1},1]$ whose mass is ${p-1\over p+1}$ at ${1\over 2}$ and is uniformly distributed otherwise.
\end{theorem}

The method for these investigations is to study asymptotics of the actual points underlying the construction of the Newton polygons for the ghost series. The extra flexibility of having a power series in hand allows one to establish results like Theorem \ref{thm:gouvea-dist} without the annoying combinatorics that would underlie proving an exact Buzzard--Gouv\^ea bound holds.

\begin{remark}\label{remark:gouvea-mazur-remark}
We also explored the relationship between the ghost series and the Gouv\^ea--Mazur conjecture \cite[Conjecture 1]{GouveaMazur-FamiliesEigenforms}. Namely, one might ask if one sees ``logarithmic-sized ghost families'' as suggested by Buzzard's conjecture. Indeed, we do.

In the discussion of the ghost spectral halo, we observed that all the zeros of the ghost coefficients occur at integer weights. Moreover, the set of zeros of a given coefficient is a linear function of its index. For example, if $k\geq 2$ is an integer then the zeros of the coefficients $g_1(w),\dotsc,g_{d_k}(w)$ (over the component containing $k$) are completely contained in the list $2,4,\dotsc,k-2$. In particular, if $v_p(w_\kappa - w_k) \geq 1 + \ceil{\log_p(k)}$ then $v_p(g_i^{(\varepsilon)}(w_\kappa)) = v_p(g_i^{(\varepsilon)}(w_k))$ for $1 \leq i \leq d_k$ and so
\begin{equation*}
\kappa \mapsto \NP(G_{\kappa}^{\leq d_k})
\end{equation*}
is constant on $v_2(w_{\kappa} - w_{k}) \geq 1 + \ceil{\log_p(k)}$. For example, $S_{62}(\SL_2 \Z)$ is four-dimensional with $T_2$-slopes $6,6,14,14$ and Figure \ref{fig:gen-slopes-2} illustrates these are the lowest four ghost slopes on $v_2(w_\kappa - w_{62}) \geq 7 = 1 + \ceil{\log_2(62)}$.
\end{remark}

\subsection{Halos and arithmetic progressions}\label{subsec:intro-global-halo}
We turn now towards one consequence of the ghost conjecture.  For $\kappa \in \cal W$, let us write $\alpha_{\kappa} := \sup_{w \in \Z_p} v_p(w_{\kappa} - w)$. Since the zeros of the ghost coefficients are all integers, it is easy to see that if $\kappa,\kappa'$ lie on the same component and $v_p(w_{\kappa'}-w_{\kappa}) > \alpha_{\kappa}$, then $\NP(G_{\kappa'}) = \NP(G_{\kappa})$.  In particular, if $w_\kappa \nin \Z_p$, then $\alpha_\kappa$ is finite and there is a small disc around $w_\kappa$ on which the ghost slopes are all constant.

The simplest example is to fix $r \geq 0$ an integer and $v$ a rational number $r < v < r+1$. Then $\kappa \mapsto  \NP(G_{\kappa})$ is constant on the disc $v_p(w_\kappa) = v$, and the Newton polygons scale linearly with $v$, forming ``halos''. We've illustrated the halos in Figure \ref{fig:gen-slopes} below where we've plotted the first twenty slopes on $v_p(w_{\kappa}) = v$ for $v \nin \Z$ when $p=2$ and $N = 1$. (The omitted regions are indicated with an open circle.\footnote{We stress that the behavior of the slopes in the omitted regions may be very complicated, interweaving the disjoint branches that we've drawn.}) Note the picture over $v_2(w_\kappa) < 3$ illustrates the result of Buzzard--Kilford \cite{BuzzardKilford-2adc}. Over $3 < v < 4$  you see pairs of parallel lines which hints at extra structure in the set of slopes.

\begin{figure}[htbp]
\begin{center}
\includegraphics[scale=.7]{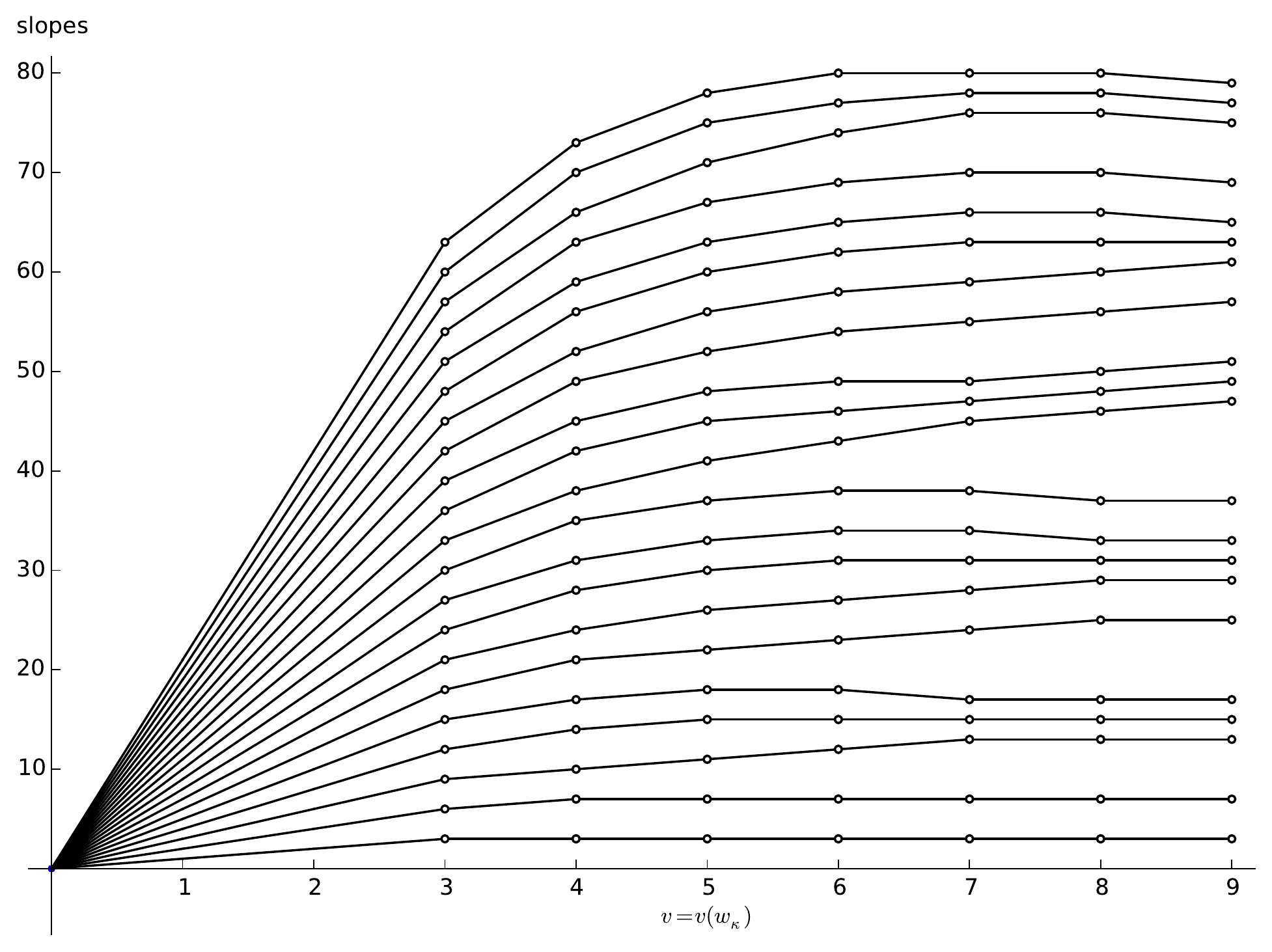}
\caption{``Halos'' for $p=2$ and $N=1$.}
\label{fig:gen-slopes}
\end{center}
\end{figure}

The following theorem explains this regularity. If $r \geq 0$, write 
\begin{equation*}
C_{p,N,r} = {p^{r+1}(p-1)(p+1)\mu_0(N)\over 24}.
\end{equation*}

\begin{theorem}[Theorem \ref{theorem:global-halo-progressions}, Remark \ref{rmk:p=2-aps-remark}]\label{theorem:global-halo}
Let $p$ be odd and assume $w_\kappa \nin \Z_p$. Write $r = \floor{\alpha_{\kappa}}$. Then the slopes of $\NP(G_{\kappa})$ form a finite union of $C_{p,N,r}$-many arithmetic progressions with the same common difference
\begin{equation*}
{(p-1)^2\over 2}\left(\alpha_\kappa + \sum_{v=1}^r (p-1)p^{r-v} \cdot v\right)
\end{equation*}
except for finitely many possibly exceptional slopes.

If $p = 2$ and $N=1$ then the same is true with $C_{2,1,r} = \max(2^{r-2},1)$ and common difference $\alpha_k + \sum_{v=3}^r 2^{r-v}\cdot v$. 
\end{theorem}
The condition $\floor{\alpha_\kappa} = 0$ is equivalent to $0 < v_p(w_\kappa) < 1$ in which case $\alpha_{\kappa} = v_p(w_\kappa)$. Thus, Theorem \ref{theorem:global-halo} generalizes Theorem \ref{theorem:intro-aps-boundary}.

Note that Theorem \ref{theorem:global-halo} applies to $p$-adic annuli $r < v_p(w_{\kappa} - w_{k_0}) < r+1$ for any integer $k_0$, and $\kappa\mapsto \NP(G_\kappa)$ is constant on each fixed radius $v_p(w_\kappa - w_{k_0}) = v \in (r,r+1)$. Thus the halo behavior is stable under re-centering the coordinate $w$ at any integral weight. We illustrate this in Figure \ref{fig:gen-slopes-2} below, showing the halos near the weight $w=w_{62}$ when $p=2$ and $N=1$.  (The interested reader may want to compare Figure \ref{fig:gen-slopes-2} to the discussion in the last paragraph of \cite[Section 3]{Buzzard-SlopeQuestions}.)

\begin{figure}[htbp]
\begin{center}
\includegraphics[scale=.7]{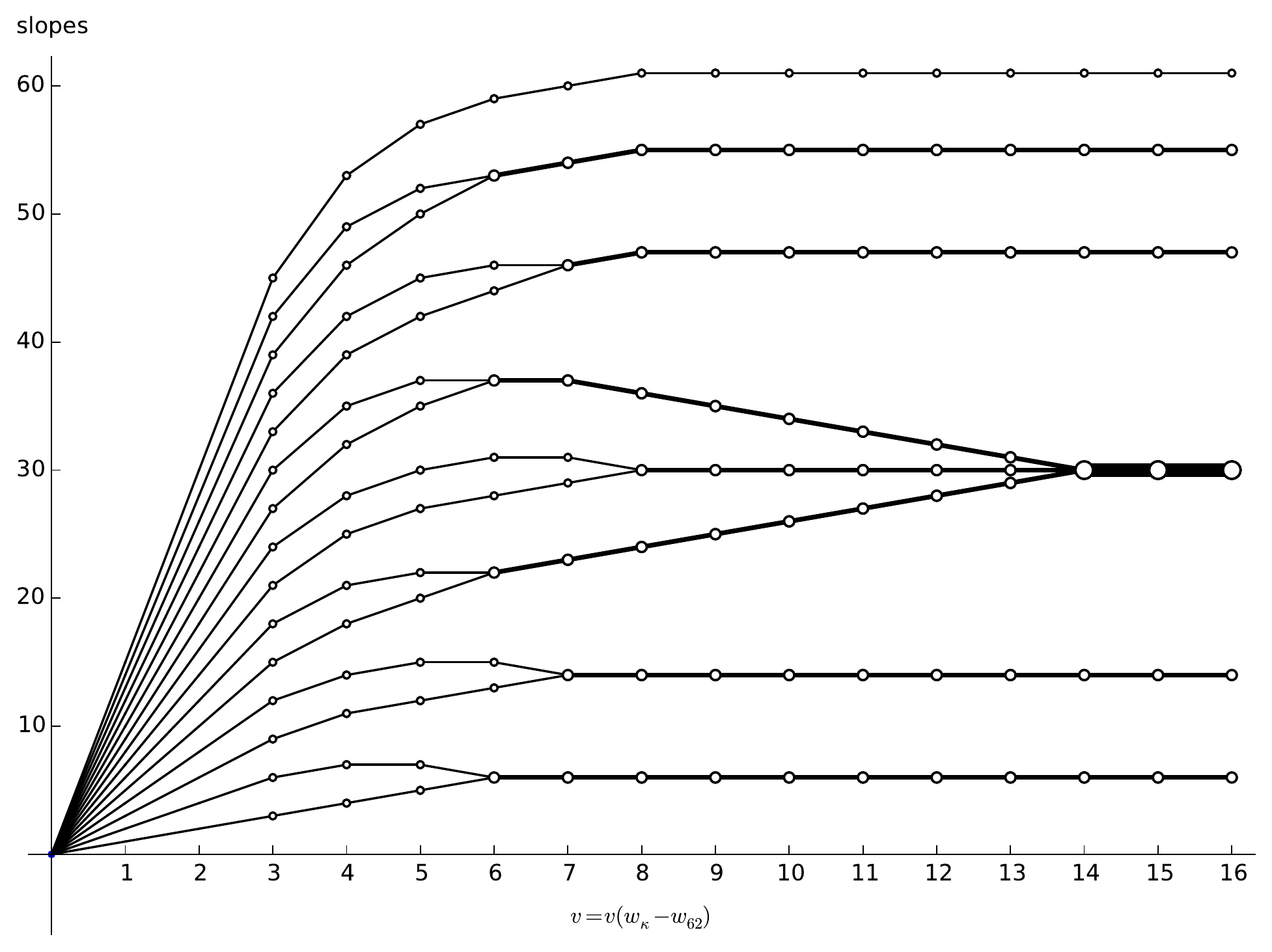}
\caption{``Halos'' centered at the weight $w = w_{62}$ when $p=2$ and $N =1$.}
\label{fig:gen-slopes-2}
\end{center}
\end{figure}

There are several interesting observations regarding Figure \ref{fig:gen-slopes-2}. First, if $v_2(w_\kappa - w_{62}) > 3$ then $v_2(w_\kappa) = 3$, so the picture in Figure \ref{fig:gen-slopes-2} is nearly completely contained within the omitted regions in Figure \ref{fig:gen-slopes} over $v_2(w_\kappa) = 3$. Second, we've drawn some lines in Figure \ref{fig:gen-slopes-2} thicker than others:\ the thickness of a line corresponds to the multiplicity of a slope. On $v_2(w_\kappa-w_{62}) > 6$, we see a double slope 6; on $v_2(w_\kappa - w_{62}) > 7$, we see two 14s; and so on. Compare with the example at the end of Remark \ref{remark:gouvea-mazur-remark}. Next, the thickest line is six slope 30 families:\ these should correspond under the ghost conjecture to  six families of $p$-adic eigenforms converging to the six newforms of weight $62$ (which have slope ${62-2\over 2} = 30$). Finally, the lone family at the top of Figure \ref{fig:gen-slopes-2} is a slope $61$ family which should be thought of under the ghost conjecture as converging to the critical slope Eisenstein series of weight 62.

If the ghost conjecture is true, there are halos for $U_p$-slopes, and the slopes of $\NP(P_{\kappa})$ satisfy Theorem \ref{theorem:global-halo}. Over the annulus $0 < v_p(w_{\kappa}) < 1$, one can observe this empirically by computing classical spaces of cuspforms of weight with character of large $p$-power conductor. However, everything is much more mysterious over a $p$-adic annulus $r < v_p(w_{\kappa}) < r+1$ once $r \geq 1$:\ there are no locally algebraic weights in that region and thus no classical spaces of cuspforms.

\subsection{Irregular cases}
The basic heuristic in the ghost series construction is that the zeros of the coefficients of Fredholm series give rise to repeated slopes and that newforms provide many repeated slopes. In a separate article (\cite{BergdallPollack-FractionalSlopes}) we show that non-integral, and thus repeated, slopes always appear when $p$ is an odd $\Gamma_0(N)$-irregular prime. One could hope that careful predictions of where these fractional slopes appear could lead to a modification of the ghost series which would work in any case.

We examined carefully the case where $p=59$ and $N=1$ and came up with a way to modify infinitely many, relatively sparse, coefficients by adding a new zero.  We tested our modified ghost series against the $U_{59}$-slopes for weights $k \leq 1156$ and they matched perfectly. However, computing actual slopes is computationally difficult and we feel we do not have enough data to support an actual conjecture.\footnote{And, we cannot compare to Buzzard's algorithm since it doesn't compute slopes in  irregular cases.} 

The precise indices where the zeros are added and the precise zeros which are added are determined by the list of slopes in weight two spaces with character of conductor $p=59$ (some of these are fractional).  It would be interesting to have a modification which works for general $p$ and $N$ (after computing this finite amount of data). Moreover, such a modification would hopefully be regular enough and sparse enough so that the results of Sections \ref{subsec:distribution} and \ref{sec:global} will go through in the general case.

\subsection{Organization}
Section \ref{section:explicit-analysis} is concerned with explicitly determining information about the ghost series, including proving that it is entire in the variable $t$. However, the reader may want to skip directly to Section \ref{section:comparison} where we give more precise information when $p=2$ and $N=1$. This section also contains the bulk of the numerical evidence for the ghost conjecture. Sections \ref{subsec:distribution} deals with asymptotics of ghost slopes. It relies heavily on Section \ref{section:explicit-analysis}. The same is true for Section \ref{sec:global}, where we describe the halos and discuss the arithmetic properties of ghost slopes. Section \ref{sec:modification} contains a modification of the ghost series when $p=2$. The main theme is dealing with fractional slopes that appear in certain spaces.

\subsection{Conventions}\label{subsec:notations}
We maintain all the notations presented in the introduction. We also make the following conventions.

If $P(t) = 1 + \sum a_i t^i \in \C_p[[t]]$ is an entire series then we write $\NP(P)$ for its Newton polygon, which is the lower convex hull of the set of points $\set{(i,v_p(a_i)) \st a_i\neq 0}$. The {\em slopes} of $P$ are the slopes of $\NP(P)$. The {\em $\Delta$-slopes} of $P$ are the differences $v_p(a_i) - v_p(a_{i-1})$ for $i=1,2,\dotsc$ with $a_{i},a_{i-1}\neq 0$, i.e.\ the slopes of the line segment connecting consecutive points before taking the Newton polygon. When $P$ is the Fredholm series for $U_p$ we will use $U_p$-slopes and when $P$ is the ghost series we will say ghost slopes and ghost $\Delta$-slopes.

If $f(x), g(x)$ and $h(x)$ are real-valued functions of a variable $x = (x_1,\dotsc,x_n) \in \R^n$ then we write
\begin{equation*}
f(x) = g(x) + O(h(x))
\end{equation*}
to mean that there exists an $M \geq 0$ and a constant $A > 0$ such that $\abs{f(x) - g(x)} \leq A\abs{h(x)}$ whenever $\norm{x} \geq M$ (where $\norm{-}$ is the standard norm on $\R^n$). 
If $h_1(x_1),\dotsc,h_n(x_n)$ are $n$ functions on a single variable then we write $f(x) = g(x) + O(h_1(x_1),\dotsc,h_n(x_n))$ to mean the above with $h(x) := \sup_{i} h(x_i)$.

\subsection*{Acknowledgements} We thank Kevin Buzzard and Liang Xiao for helpful discussions. The first author was supported by NSF grant DMS-1402005 and the second author was supported by NSF grant DMS-1303302.

\section{Explicit analysis of the ghost series}\label{section:explicit-analysis}

We fix a prime $p$, a tame level $N$, and an even Dirichlet character $\varepsilon$ of conductor $p$. We write $g_i = g_i^{(\varepsilon)}$ for the coefficients of the $p$-adic ghost series $G(w,t)=G^{(\varepsilon)}(w,t)$ of tame level $N$ over the component $\cal W_{\varepsilon}$. We have two goals in this section. First, we will prove that $G(w,t)$ is entire over $\Z_p[[w]]$ (see Proposition \ref{proposition:ghost-entire}). Second, we will show that if the ghost conjecture is true then either $p$ is an odd $\Gamma_0(N)$-regular prime or $p=2$ and $N=1$ (see Theorem \ref{thm:ghost-true-regular}). 

We begin by recalling that if $k \in \cal W_{\varepsilon}$ then
\begin{equation}\label{eqn:zeros}
g_i(w_{k}) = 0 \iff d_k < i < d_k + d_k^{\new}.
\end{equation}
Throughout this section, we also refer to the integer $k$ as the zero of $g_i$ when we truthfully mean the coordinate $w_{k}$.
\begin{lemma}
\label{lemma:arithprogs}
If $N>1$ or $p>3$, then the zeros of $g_i$ are integers $k$ which form a finite arithmetic progression with common difference $p-1$, if $p$ is odd, and $2$ if $p$ is even.
\end{lemma}
\begin{proof}
By \eqref{eqn:zeros} it suffices to show that 
\begin{equation*}
d_k < i < d_k + d_k^{\new} \implies \text{one of } \begin{cases}
d_{k+\varphi(2p)} < i < d_{k+\varphi(2p)} + d_{k+\varphi(2p)}^{\new}\\
i \leq d_{k+\varphi(2p)}.
\end{cases}
\end{equation*}
But if $N > 1$ or $p > 3$ then $d_k \leq d_{k + \varphi(2p)}$ and $d_{k}^{\new} \leq d_{k+\varphi(2p)}^{\new}$ (Lemma \ref{lemma:dim-formulas-increasing}). Thus if the first possibility fails, it must be due to the lower inequality, which is what we wanted to show.
\end{proof}

\begin{remark}
Lemma \ref{lemma:arithprogs} only misses $N=1$ and $p=2,3$. When $N=1$ and $p=2$, the zeros of $g_i$ are 
$$
6i+8,6i+10,\dots,12i-4,12i-2,12i+2,
$$
and thus are just missing the single term $12i$ in an arithmetic progression.  Similarly, for $N=1$ and $p=3$, the zeros of  $g_i$ are
$$
4i+6,4i+8,\dots,12i-4,12i-2,12i+2.
$$
(See Proposition \ref{proposition:explicit-p=2-info} and Table \ref{table:least-zeros}.)
\end{remark}
For each coefficient $g$, write $\HZ(g)$ (resp.\ $\LZ(g)$) for the highest (resp.\ lowest) $k$ such that $w_{k}$ is a zero of $g$. The following proposition describes these highest and lowest zeros up to constants bounded independent of $i$.

\begin{proposition}\label{prop:approximate-zero-location}
As functions of $i$,
$$
\HZ(g_i) = \frac{12i}{\mu_0(N)} + O(1) ~\text{~~~and~~~}~
\LZ(g_i) = \frac{12i}{\mu_0(N)p} +O(1)
$$
\end{proposition}

\begin{proof}
By standard dimension formulas (see Appendix \ref{app:dimension-formulae}), we have that 
$d_k = \frac{k \mu_0(N)}{12} +O(1)$ and $d_k^{\new} = \frac{k \mu_0(N)(p-1)}{12} + O(1)$.
Thus, the largest $k$ satisfying $d_k < i$ equals $\frac{12i}{\mu_0(N)} +O(1)$, and the smallest $k$ satisfying $i < d_k + d_k^{\new}$ is $\frac{12i}{\mu_0(N)p} +O(1)$. The proposition follows from the definition \eqref{eqn:zeros}.
\end{proof}

We also explicitly describe how the zeros of the coefficients and their multiplicities change as we increase indices. Write $\Delta_i(w) = g_i(w)/g_{i-1}(w)$. The definition of the multiplicity patterns $m_i(-)$ in Section \ref{subsec:ghost-conjecture} implies that $\Delta_i$ has only simple zeros and poles at some finite set of $w = w_{k}$. More specifically, if $k \in \cal W_{\varepsilon}$ then
\begin{equation}\label{eqn:Deltaizeros}
\Delta_i(w_{k}) = 0 \iff m_i(k) = m_{i-1}(k) +1 \iff d_k+1 \leq i \leq d_k + \bfloor{{d_k^{\new}\over 2}}
\end{equation}
and
\begin{equation}\label{eqn:Deltaipoles}
\Delta_i(w_{k}) = \infty \iff m_i(k) = m_{i-1}(k) - 1 \iff d_k + \bfloor{{d_k^{\new}-1\over 2}} +2 \leq i \leq d_k + d_k^{\new}.
\end{equation}
For notation, we will always write $\Delta_i = \Delta_i^+/\Delta_i^-$ in lowest common terms. Thus $\Delta_i^{\pm} \in \Z[w]$ and the zeros are of the form $w_{k}$ with $k \in \Z$. We write $\HZ(\Delta_i^{\pm})$ and $\LZ(\Delta_i^{\pm})$ for the highest and lowest zeros as with $g_i$ above.
\begin{proposition}
\label{prop:change_in_zeroes}
\leavevmode
\begin{enumerate}
\item \label{change:b}
The zeros of $\Delta_i^-$ form an arithmetic progression with common difference $p-1$ if $p$ is odd and $2$ if $p=2$.  The same is true for the zeros of $\Delta_i^+$ unless $N=1$ and $p=2, 3$.
\item \label{change:c} $\displaystyle \HZ(\Delta_i^+) = \frac{12i}{\mu_0(N)}+O(1)$ and 
$\displaystyle \LZ(\Delta_i^+) = \frac{24i}{\mu_0(N)(p+1)} + O(1)$.
\item \label{change:d} $\displaystyle \HZ(\Delta_i^-) = \frac{24i}{\mu_0(N)(p+1)}+O(1)$ and
$\displaystyle \LZ(\Delta_i^-) = \frac{12i}{\mu_0(N)p} + O(1)$.
\end{enumerate}
\end{proposition}

\begin{proof}
Part (\ref{change:b}) follows from \eqref{eqn:Deltaizeros} and \eqref{eqn:Deltaipoles} together with Lemma \ref{lemma:dim-formulas-increasing} (as in the proof of Lemma \ref{lemma:arithprogs}).  Parts (\ref{change:c}) and (\ref{change:d}) follow similarly as in the proof of Proposition \ref{prop:approximate-zero-location}.
\end{proof}

\begin{remark}
In Proposition \ref{proposition:explicit-p=2-info} and Table \ref{table:least-zeros}, we give formulas making the above $O(1)$-terms precise when $N=1$ and $p=2,3,5$, and $7$. The qualification for $p=2,3$ and $N=1$ will be inconsequential as we move forward (see the proofs of Lemma \ref{lemma:zeropole_count} and Proposition \ref{prop:ith_newtonslope}, for example).
\end{remark}

For a non-zero element $\Delta \in \Z_p[[w]]$, we write $\lambda(\Delta)$ for the number of zeros of $\Delta$ in the open disc $v_p(w) >0$. We extend this to the field of fractions in the obvious way.

\begin{lemma}
\label{lemma:zeropole_count}
If $p > 2$ then
\begin{enumerate}
\item $\lambda(\Delta_i^+) = \displaystyle \frac{12i}{\mu_0(N)(p+1)} + O(1)$ and
\item $\lambda(\Delta_i^-) = \displaystyle \frac{12i}{\mu_0(N)p(p+1)} + O(1)$.
\end{enumerate}
For $p=2$, the same formulas hold if we replace the 12 by a 6.
\end{lemma}

\begin{proof}
For $p>2$, the number of zeros of $\Delta_i^+$ equals 
$$
\frac{\HZ(\Delta_i^+) - \LZ(\Delta_i^+)}{p-1}+1 = \frac{1}{p-1} \left(  \frac{12i}{\mu_0(N)} - \frac{24i}{\mu_0(N)(p+1)} \right) + O(1) = \frac{12i}{\mu_0(N)(p+1)} + O(1)
$$
by Proposition \ref{prop:change_in_zeroes}(\ref{change:b},\ref{change:c}).\footnote{The $O(1)$ term absorbs the qualification that Proposition \ref{prop:change_in_zeroes} isn't quite true if $p=2,3$ and $N=1$.}  The zeros of $\Delta_i^-$ and $p=2$ is done similarly.
\end{proof}

\begin{remark}\label{remark:lambdaDeltai}
The difference $\lambda(\Delta_i) = \lambda(\Delta_i^+) - \lambda(\Delta_i^-)$ equals the $i$-th $\Delta$-slope of the mod $p$ reduction $\overline{G(w,t)}\in \F_p[[w,t]]$ (with the $w$-adic valuation on $\F_p[[w]]$). By Lemma \ref{lemma:zeropole_count}, the $i$-th $\Delta$-slope equals $ \frac{12i(p-1)}{\mu_0(N)p(p+1)}$ up to a bounded constant.  We return to this in Section \ref{sec:global}.
\end{remark}

Recall that if $R$ is a local ring with maximal ideal $\ideal m$ and $F(t) = \sum r_i t^i \in R[[t]]$ then $F$ is called entire if there exists a sequence of integers $c_i$ such that $r_i \in \ideal m^{c_i}$ and $c_i / i \goto \infty$. If $R = \Z_p[[w]]$ and $G(w,t)\in \Z_p[[w,t]]$ is entire over $\Z_p[[w]]$ then the specialized series $G(w',t) \in \C_p[[t]]$ is entire in the usual sense for all $w' \in \C_p$ with $v_p(w') > 0$ (see \cite[Section 1.3]{ColemanMazur-Eigencurve}).

\begin{proposition}\label{proposition:ghost-entire}
The ghost series $G^{(\varepsilon)}(w,t)$ is entire series over $\Z_p[[w]]$. In particular, if $\kappa \in \cal W$ then $G_\kappa(t)$ is an entire series.
\end{proposition}
\begin{proof}
Every root $w_{k}$ of $g_i$ lies in $p\Z_p$, and so $g_i \in (p,w)^{\lambda(g_i)}$. We claim $\lambda(g_i)/i \goto \infty$ as $i \goto \infty$. To show the claim, it is enough to show
 \begin{equation}\label{eqn:need-lim-inf}
\liminf_i \left(\lambda(g_{i}) - \lambda(g_{i-1})\right) = \infty.
\end{equation} 
But $\lambda(g_{i})-\lambda(g_{i-1}) = \lambda(\Delta_i)$, so \eqref{eqn:need-lim-inf} follows from Lemma \ref{lemma:zeropole_count} and the remark following it.
\end{proof}

We now turn to showing that for the ghost conjecture to be true, either $p=2$ and $N=1$ or $p$ is an odd $\Gamma_0(N)$-regular prime. In addition to our running notation $d_k$ and $d_k^{\new}$, we now also write $d_k^{\ord}$ for the dimension of the $p$-ordinary subspace of $S_k(\Gamma_0(Np))$. Hida theory implies that $d_k^{\ord}$ depends only on the component $\cal W_{\varepsilon}$ containing $k$. If $\kappa \in \cal W$, write $d_{G_\kappa}^{\ord}$ for the multiplicity of the slope zero in $\NP(G_\kappa)$. We leave the following proof to the reader. (The supremums are finite by Proposition \ref{prop:approximate-zero-location}.)
\begin{lemma}[Ghost Hida theory]\label{lemma:ghost-hida}
The function $\kappa \mapsto d_{G_\kappa}^{\ord}$ is constant on connected components of $\cal W$. Specifically, if $\kappa \in \cal W_{\varepsilon}$ then 
\begin{multline*}
d_{G_\kappa}^{\ord} = \sup \set{i \st g_i^{(\varepsilon)}(w) = 1} = \sup \set{i \st m_i(k) = 0 \text{ for all $k \in \cal W_{\varepsilon}$}}\\ \geq \min \set{d_k \st k\geq 2 \text{ and } k \in \mathcal W_{\varepsilon}}.
\end{multline*}
\end{lemma}

\begin{lemma}\label{lemma:explicit-ghost-ordinary-dims}
Let $p \neq 2$.
\begin{enumerate}
\item If $4 \leq k \leq p-1$ is an even integer then $d_{G_k}^{\ord} \geq d_k$.
\item $d_{G_2}^{\ord} \geq d_2 + d_2^{\new} = d_{2 + (p-1)}$.
\end{enumerate}
\end{lemma}
\begin{proof}
First assume that $4 \leq k \leq p-1$ (so $p>3$). By Lemma \ref{lemma:dim-formulas-increasing},  $n \mapsto d_{k+n(p-1)}$ is weakly increasing with respect to $n\geq 0$, so Lemma \ref{lemma:ghost-hida} proves $d_k = \min_n d_{k+n(p-1)}\leq d_{G_k}^{\ord}$.

For part (b), Lemma \ref{lemma:weight2-inequality} implies that $d_2 + d_2^{\new} = d_{2+(p-1)}$. If $p = 3$ and $N=1$ then $d_{2+(p-1)} = \dim S_4(\SL_2 \Z) = 0$, so (b) is trivial. If $p > 3$ or $N > 1$ then Lemma \ref{lemma:dim-formulas-increasing} applies and $d_2+d_2^{\new} = d_{2 + (p-1)} = \min_{n\geq 1} d_{2 + n(p-1)}$. So, the coefficient $g_i$ at index $i = d_{2+(p-1)}$ is trivial, showing $d_{G_2}^{\ord} \geq d_{2 + (p-1)}$ by Lemma \ref{lemma:ghost-hida}.
\end{proof}

\begin{remark}
When $p$ is odd, Lemma \ref{lemma:explicit-ghost-ordinary-dims}(b) implies that one could remove $w_{2}$ as a root of any of the coefficients of the ghost series. We actually do that in Section \ref{sec:global} below (see Lemma \ref{lemma:remove-wt2-zero}).
\end{remark}

\begin{lemma}\label{lemma:p=2-ghost-bound}
If $p=2$ then $d_4\leq d_{G_2}^{\ord}$.
\end{lemma}
\begin{proof}
By Lemma \ref{lemma:ghost-hida} it suffices to show that $g_{d_4} = 1$. Since $d_4 \leq d_{2m}$ for all $m\geq 2$ (Lemma \ref{lemma:dim-formulas-increasing} if $N>1$ and trivial if $N=1$), the only possible zero for $g_{d_4}$ is $w = w_2$. But by Lemma \ref{lemma:appendix-2adic-wt-2}, $d_2+d_2^{\new} \leq d_4$ and so the last index where $w_2$ is possibly a zero is strictly less than $d_4$.
\end{proof}

\begin{theorem}\label{thm:ghost-true-regular}
Suppose the ghost conjecture is true.
\begin{enumerate}
\item If $p$ is odd then $p$ is $\Gamma_0(N)$-regular.
\item If $p = 2$ then $N = 1$.
\end{enumerate}
\end{theorem}
\begin{proof}
Let $p$ be odd and assume the ghost conjecture is true. To show that $p$ is $\Gamma_0(N)$-regular we need to show that $d_k^{\ord} = d_k$ for $k=4,\dotsc,p+1$, and we have $d_k^{\ord} \leq d_k$ in general. Since we are assuming the ghost conjecture we have $d_{G_k}^{\ord} = d_k^{\ord}$ and thus Lemma \ref{lemma:explicit-ghost-ordinary-dims} implies $d_k \leq d_{G_k}^{\ord} = d_k^{\ord} \leq d_k$. Thus we get equality throughout, proving (a).

Now let $p=2$, and assume the ghost conjecture is true. First suppose that $N \neq 1,3,7$ and we will get a contradiction. If $N\neq 1,3,7$ then Lemma \ref{lemma:appendix-2adic-wt-2} implies that $d_4 > d_2 + d_2^{\new} \geq d_4^{\ord}$ (the final inequality by Hida theory).  But if the ghost conjecture is true then Lemma \ref{lemma:p=2-ghost-bound} implies $d_4^{\ord} = d_{G_2}^{\ord} \geq d_4$, which is a contradiction. To finish the theorem, we show in Example \ref{example:issue-p=2} below that the ghost conjecture is false when $p=2$ and $N=3,7$.
\end{proof}

\begin{example}\label{example:issue-p=2}
Let $N = 3$. Then the $2$-adic ghost series  begins
\begin{equation*}
G(w,t) = 1 + t + (w-w_{8})t^2 + (w-w_{8})(w-w_{10})t^3 + \dotsb
\end{equation*}
so if the the ghost conjecture is true then there is at least one ordinary form appearing in $S_4(\Gamma_0(3))$. This is absurd since $S_4(\Gamma_0(3))$ is a zero-dimensional vector space. 

Similarly, if $N = 7$ then the $2$-adic ghost series begins
\begin{equation*}
G(w,t) = 1 + t + (w-w_{4})t^2 + t^3 + \dotsb
\end{equation*}
and so the ghost conjecture would imply that there exists a least three ordinary forms appearing in $S_4(\Gamma_0(7))$, which is only a one-dimensional space.
\end{example}

\begin{remark}
In Example \ref{example:issue-p=2}, the number of ordinary forms predicted by the ghost series doesn't even match the correct dimension of a weight four space. When $p=2$ and $N = 23$ the ghost conjecture is false, but for more subtle reasons:\ here the ghost series begins
\begin{equation*}
G(w,t) = 1 + t + t^2 + t^3 + t^4 + t^5 + (w-w_{6})t^6 + \dotsb
\end{equation*}
and there are no more trivial terms up to $t^{20}$ at least. One could even prove $d_{G_4}^{\ord}=5$ and in this case $S_4(\Gamma_0(23))$ happens to be five-dimensional. But, the slopes of $U_2$ acting on $S_4(\Gamma_0(23))$ are $\set{0,0,0,1,1}$ and so there are actually only three ordinary forms.
\end{remark}

\section{Comparison with known or conjectured lists of slopes}\label{section:comparison}

This section is devoted to proving the ghost conjecture is true in every case mentioned in Theorem \ref{theorem:intro-actual-truth} (where the $U_p$-slopes have been previously determined). We do this by determining the ghost slopes in each case. We also prove that the ghost conjecture implies a conjecture of Buzzard and Calegari on slopes of overconvergent $2$-adic cuspforms, and we derive formulas for the ghost slopes at the weight $\kappa=0$ for $p=3,5$ and $N=1$. 

We focus first on $p=2$. So, until after the proof of Theorem \ref{theorem:agree-with-buzzard-calegari} below, we write $G(w,t) = 1 + \sum g_i(w)t^i \in \Z_2[[w,t]]$ for the $2$-adic tame level 1 ghost series. The reader may freely check the first four terms are:
\begin{multline}\label{eqn:2-adic-ghost-two-terms}
G(w,t) = 1 + (w-w_{14})t + (w-w_{20})(w-w_{22})(w-w_{26})t^2 + \\
 (w-w_{26})(w-w_{28})(w-w_{30})(w-w_{32})(w-w_{34})(w-w_{38})t^3 + \dotsb
\end{multline}
Recall we write $\Delta_i = g_i/g_{i-1}$ and in lowest terms $\Delta_i = \Delta_i^+/\Delta_i^-$.
\begin{proposition}\label{proposition:explicit-p=2-info}
Let $p=2$ and $N=1$.
\begin{enumerate}
\item \label{explicit:a}
$g_i(w_{k}) = 0$ if and only if $k$ is an even integer among $\set{6i + 8 , \dotsc 12i - 2} \union \set{12i+2}$.
\item 
\label{explicit:b}
If $i\geq 1$ then:
\begin{enumerate}[(i)]
\item The zeros of $\Delta_i^+$ are $w_{k}$ where $k=8i+4, \dotsc, 12i - 2, 12i + 2$ is even.
\item The zeros of $\Delta_i^-$ are $w_{k}$ where $k = 6i + 2, \dotsc, 8i-2$ is even.
\end{enumerate}
\end{enumerate}
\end{proposition}
\begin{proof}
We check (\ref{explicit:a}), leaving the remainder to the reader.  First note that $d_2 = d_2^{\new} = 0$ so $w_2$ does not occur as a zero.  Further, if $k\geq 4$ is an even integer then $d_{k+12} = d_k + 1$ and $d_{k+12}^{\new} = d_k^{\new} + 1$ (as follows easily from Appendix \ref{app:dimension-formulae}). By  \eqref{eqn:zeros}, part (a) follows from:
\begin{claim}
If $i\geq 1$ then for all even $k\geq 4$,
\begin{align}
d_k + 1 \leq i &\iff k \leq 12i + 2 \text{ and $k\neq 12i$, and}\label{eqn:formula-to-induct-1}\\
i \leq d_k + d_k^{\new} - 1 &\iff  6i + 8 \leq k.\label{eqn:formula-to-induct-2}
\end{align}
\end{claim}
To prove \eqref{eqn:formula-to-induct-1} and \eqref{eqn:formula-to-induct-2}, we work inductively. Namely, if the inequalities on either side of \eqref{eqn:formula-to-induct-1} are true for $(i,k)$ then they are also true $(i+1,k+12)$ and the if inequalities on either side of \eqref{eqn:formula-to-induct-2} are true for $(i,k)$ then they are also true for $(i+2,k+12)$. By induction on $i$, it is enough to prove the claim for $i=1,2$, which is done by examination of \eqref{eqn:2-adic-ghost-two-terms}.
\end{proof}

\begin{theorem}\label{theorem:agree-BK-p=2}
Let $p=2$ and $N = 1$.
\begin{enumerate}
\item 
\label{BK2:a}
If $i\geq 0$ then $\lambda(g_i) = {i + 1 \choose 2}$.
\item If $v_2(w_{\kappa}) < 3$ then the slopes of $\NP(G_\kappa)$ are $\set{ j\cdot v_2(w_{\kappa}) \st j \geq 1}$ and $\NP(G_\kappa) = \NP(P_\kappa)$. 
\end{enumerate}
\end{theorem}
\begin{proof}
We first prove part (a). The case of $i = 0$ is trivial since $g_0(w) = 1$. If $i\geq 1$ then Proposition \ref{proposition:explicit-p=2-info}(b) implies that
\begin{align*}
\lambda(g_i) &= \lambda(g_{i-1}) + \sizeof \set{\text{even } 8i+4,\dotsc,12i-2,12i+2} - \sizeof \set{\text{even } 6i + 2, \dotsc, 8i-2}\\
&= \lambda(g_{i-1}) + (2i - 1) - (i-1) 
= \lambda(g_{i-1}) + i.
\end{align*}
Thus, $\lambda(g_i) = {i+1\choose 2}$ by induction.
It follows from the ghost spectral halo (Theorem \ref{thm:gsh}) and part (\ref{BK2:a}) that if $v_2(w_\kappa) < 3$ then $\NP(G_{\kappa})$  is equal to the lower convex hull of the set of points $(i, {i+1\choose 2}v_2(w_\kappa))$, whose slopes are easily checked to be $v_2(w_\kappa), 2v_2(w_\kappa),\dotsc$. This is precisely the list of $U_p$-slopes on $v_2(w_\kappa)<3$ computed by Buzzard and Kilford in \cite[Theorem B]{BuzzardKilford-2adc}.
\end{proof}

Let's now compare the $2$-adic ghost series with actual Fredholm series at negative even integers (following Buzzard and Calegari \cite{BuzzardCalegari-2adicSlopes}).

\begin{theorem}\label{theorem:agree-with-buzzard-calegari}
Let $p=2$ and $N=1$. If $k\leq 0$ is an even integer and $i\geq 1$ then 
\begin{equation*}
v_2(g_i(w_{k})) = v_2\left(\prod_{j=1}^i 2^{2j}{(-k+12j+2)!(-k+6j)!\over (-k+8j+2)!(-k+8j-2)!(-k + 12j)}\right)
\end{equation*}
In particular, the slopes of $\NP(G_k)$ agree with the slopes predicted by Buzzard and Calegari in \cite[Conjecture 2]{BuzzardCalegari-2adicSlopes}.
\end{theorem}
Since Buzzard and Calegari proved their conjecture for $k=0$ (\cite[Theorem 1]{BuzzardCalegari-2adicSlopes}) we deduce:
\begin{corollary}\label{corollary:bc-k=0}
If $p=2$ and $N=1$ then $\NP(G_0) = \NP(P_0)$.
\end{corollary}

\begin{proof}[Proof of Theorem \ref{theorem:agree-with-buzzard-calegari}]
Note that if $k\leq 0$ then $g_i(w_{k}) \neq 0$ for all $i\geq 0$ and thus $\Delta_i(w_{k})$ is well-defined. By induction on $i\geq 1$, it suffices to show that if $k\leq 0$ is an even integer then
\begin{equation}\label{eqn:formula-2adic-ghost-negintegers}
v_2\left(\Delta_i(w_{k})\right) = v_2\left(2^{2i}{(-k+12i+2)!(-k+6i)!\over (-k+8i+2)!(-k+8i-2)!(-k + 12i)}\right).
\end{equation}
To this end, Proposition \ref{proposition:explicit-p=2-info}(\ref{explicit:b}) implies 
\begin{align*}
v_2\left(\Delta_i(w_{k})\right) &= v_2\left({(w_{k} - w_{8i+4}) \dotsb (w_{k}-w_{12 i - 2})(w_{k} - w_{12 i + 2}) \over (w_{k} - w_{6i+2}) \dotsb (w_{k} - w_{8i-2})}\right).
\end{align*}
The $\dotsb$ indicate running over only even integers.  Since $v_2(w_{k} - w_{k'}) = 2 + v_2(k - k')$,
\begin{align*}
v_2\left(\Delta_i(w_{k})\right) 
&= v_2\left(2^{2i} {(k-(8i+4))\dotsb (k-(12i-2))(k-(12i+2))\over (k-(6i+2))\dotsb (k-(8i-2))}\right) \\
&= v_2\left(2^{2i}{(-k+12i+2)!(-k+6i)!\over (-k+8i+2)!(-k+8i-2)!(-k + 12i)}\right).\end{align*}
as desired.
\end{proof}

We now release our restriction to $p=2$ and $N=1$. Analogs of Proposition \ref{proposition:explicit-p=2-info} may be carried out for other values of $p$ and $N$. In Table \ref{table:least-zeros} below, we list the outcome for $p=3,5,7$ and tame level $N = 1$, on the weight component corresponding to $k\congruent 0 \bmod p-1$. With the details from Table \ref{table:least-zeros} available, it is easy to compute the ghost $w$-adic $\Delta$-slopes (see Table \ref{table:wadic-slopes} --- we've added a few more cases there as well).

\begin{table}[htpp]
\caption{Explicit determination of zeros of $\Delta_i^{\pm}$ for the $p$-adic tame level 1 ghost series $G(w,t) = 1 + \sum g_i(w)t^i$ on the component of weights $k\congruent 0 \bmod p-1$ for $p=3,5,7$.}
\begin{center}
\renewcommand{\arraystretch}{1.3}
\begin{tabular}{|c|c|c|c|}
\hline
$p$ & $3$ & $5$ & $7$\\
\hline
$\HZ(\Delta_i^+)$ & $12i+2$ & $12i-4$ & $12i-6$\\
\hline
$\LZ(\Delta_i^+)$ & $6i+4$ & $4i+4$ & $6\floor{{i\over 2}}$\\
\hline
$\HZ(\Delta_i^-)$ & $6i-2$ & $4i-4$ & $6\floor{{(i-1)\over 2}}$\\
\hline
$\LZ(\Delta_i^-)$ & $ 4i+2$ & $4 \floor{\frac{3i}{5}} + 4$ & $ 6 \floor{\frac{2i}{7}}+6$\\
\hline
\end{tabular}
\end{center}
\label{table:least-zeros}
\end{table}

\begin{table}[htpp]
\caption{Differences of consecutive $\lambda$-invariants.}
\begin{center}
\renewcommand{\arraystretch}{1.3}
\begin{tabular}{|c|c|c|c|c|c|c|}
\hline
$p$ & $3$ & $5$ & $5$ & $7$ & $7$ & $7$ \\
\hline
Weight component & $0\bmod 2$ & $0\bmod 4$ & $2 \bmod 4$ & $0 \bmod 6$ & $2\bmod 6$ & $4\bmod 6$\\
\hline
$\lambda(g_i)-\lambda(g_{i-1})$ & $2i$ & $\floor{{8i\over 5}}$ & $\floor{{(8i+4)\over 4}}$ & $\floor{{9i\over 7}}$ & $\floor{{(9i+6)\over 7}}$ & $\floor{{(9i+3)\over 7}}$\\
\hline
\end{tabular}
\end{center}
\label{table:wadic-slopes}
\end{table}

\begin{theorem}\label{theorem:consisten-boundary}
Suppose that $N=1$ and that $G(w,t) = 1 + \sum g_i(w)t^i$ is the ghost series on a component to be determined. Then, $\NP(G_\kappa) = \NP(P_\kappa)$ if:
\begin{enumerate}
\item If $p = 3$ and $v_p(w_\kappa) < 1$.
\item If $p=5$ and $\kappa = z^k \chi$ where $\chi$ is a primitive modulo 25 and $\chi(-1) = (-1)^k$.
\item If $p = 7$ and $\kappa = z^k\chi \in \cal W_0\union \cal W_2$ and $\chi$ is primitive modulo 49.
\end{enumerate}
\end{theorem}
\begin{proof}
The ghost $w$-adic $\Delta$-slopes in Table \ref{table:wadic-slopes} are always weakly increasing. So, for $p=3,5,7$ and $\kappa \in \cal W$ with $v_p(w_\kappa) < 1$, Theorem \ref{thm:gsh} implies that the slopes on $\NP(G_{\kappa})$ are given by $\set{(\lambda(g_i)-\lambda(g_{i-1}))\cdot v_p(w_\kappa) \st i = 1,2,\dotsc}$. The proof is then complete from Table \ref{table:wadic-slopes} once we verify these are the slopes of $\NP(P_{\kappa})$ in cases (a), (b) and (c). 

The case (a) is the main result of Roe's paper \cite{Roe-Slopes}. The case (b) is due to Kilford \cite{Kilford-5Slopes}. The case (c) was computed by Kilford and McMurdy in \cite{KilfordMcMurday-7adicslopes}.\footnote{In comparing our statement to \cite{Kilford-5Slopes,KilfordMcMurday-7adicslopes}, one is forced to unwind the various choices made in those papers regarding embeddings of cyclotomic fields into $\C_p$. Doing it carefully, one sees that \cite{KilfordMcMurday-7adicslopes} does not contain any result regarding the component of weights $k\congruent 4 \bmod 6$ when $p=7$.}
\end{proof}
One may also generalize Theorem \ref{theorem:agree-with-buzzard-calegari}. In Table \ref{table:bc-type-expressions} below, for $p=3$ and $p=5$, we give expressions that allows us to compute the Newton polygon of the ghost series at negative even integers as the Newton polygon of a series whose coefficients are rational functions involving simple factorials when $p=3$ and $p=5$.
\begin{table}[htpp]
\caption{Buzzard--Calegari-type expressions for $\NP(G_k)$ at negative integers $k \congruent 0 \bmod p-1$ when $p=3$ and $p=5$.}

\begin{center}
{\renewcommand{\arraystretch}{2}
\begin{tabular}{|c|c|}
\hline $p$ & $Q_j(k)$ such that $v_p(g_i(w_{k})) = v_2\left(\prod_{j=1}^i Q_j(k)\right)$\\
\hline
  $3$ &  $\displaystyle {3^{2j}(-k/2 + 6j + 1)!(-k/2 + 2j)!\over (-k/2 + 3j+1)!(-k/2+3j-1)!(-k/2+6j)}$\\
  $5$ & $\displaystyle {5^{\floor{8j/5}} (-k/4 + 3j-1)!(-k/4+\floor{3j/5} )!\over (-k/4+j)!(-k/4+j-1)!}$\\
  \hline
\end{tabular}
}
\end{center}
\label{table:bc-type-expressions}
\end{table}

From Table \ref{table:bc-type-expressions} we can determine the slopes of the ghost series at $\kappa = 0$ for $p=3,5$. The expressions we derive agree with those conjectured in Loeffler's paper \cite[Conjecture 3.1]{Loeffer-SpectralExpansions}.
\begin{proposition}\label{prop:loeffler-sequences}
The Newton polygon $\NP(G_0)$ for $p=3,5$ has slopes
\begin{equation*}
\begin{cases}
2i + 2v_3\left({(2i)!\over i!}\right) & \text{if $p=3$;}\\
i + 2 v_5\left({(3i)!\over i!}\right) & \text{if $p=5$.}
\end{cases}
\end{equation*}
\end{proposition}
\begin{proof}
The sequences given are increasing with respect to $i$. If we show they agree with the $\Delta$-slopes of $G_0$ then we will be done. The proof is similar in either case, so we just deal with the case $p=5$. By Table \ref{table:bc-type-expressions} we have
\begin{equation}\label{eqn:p=5loeffler-1}
v_2(\Delta_i(w_{0})) = v_5\left({5^{\floor{8i/5}} (3i-1)!(\floor{3i/5} )!\over (i)!(i-1)!}\right) = i + v_5\left({5^{\floor{3i/5}} (3i-1)!(\floor{3i/5} )!\over (i)!(i-1)!}\right)
\end{equation}
But for any integer $n\geq 1$ and prime $p$ we have $v_p(\floor{n/p}!) = v_p(n!) - \floor{n/p}$. Thus 
\begin{equation}\label{eqn:p=5loeffler-2}
v_5\left({5^{\floor{3i/5}} (3i-1)!(\floor{3i/5} )!\over (i)!(i-1)!}\right) = v_5\left({ (3i-1)!(3i)!\over (i)!(i-1)!}\right) = v_5\left({(3i)!^2\over (i)!^2}\right).
\end{equation}
Combining \eqref{eqn:p=5loeffler-1} and \eqref{eqn:p=5loeffler-2}, we deduce our claim.
\end{proof}

\begin{remark}
For $p=7$ and $N=1$ the $i$-th slope of $\NP(G_0)$ is 
\begin{equation*}
i + v_7\left({(2i)!(2i-1)!\over \floor{i/2}!\floor{(i-1)/2}!}\right) = i + 2v_7\left((2i)!\over \floor{(i-1)/2}!\right) - v_7(i) - \begin{cases}
1 & \text{if $i \congruent 0 \bmod 14$}\\
0 & \text{otherwise.}
\end{cases}
\end{equation*}
(compare with the comments of Loeffler in the final paragraph prior to Section 4 of \cite{Loeffer-SpectralExpansions}).
\end{remark}

\section{Distributions of slopes}
\label{subsec:distribution}

For a fixed integer $k$ we write $s_1(k) \leq s_2(k) \leq \dotsb$ for the slopes of $\NP(G_k)$. Recall our conventions for $O$-notation (Section \ref{subsec:notations}). Throughout this section, functions of $i$ and $k$ are restricted to $i\geq 1$ and $k\geq 2$. The main theorem of this section is:\

\begin{theorem}
\label{thm:asymptotic}
$$
s_i(k) = 
\begin{cases}
\displaystyle \frac{12i}{\mu_0(N)(p+1)} + O(\log(k), \log(i))  & \text{if~}i\leq d_k \text{~or~}i > d_k +d_k^{\new}, \\
~\\
\displaystyle \frac{k}{2} + O(\log(k)) & d_k < i \leq  d_k + d_k^{\new}.
\end{cases}
$$
\end{theorem}

Before beginning the proof of Theorem \ref{thm:asymptotic}, we state two corollaries (Theorems \ref{theorem:buzzard-gouvea-intro} and \ref{thm:gouvea-dist} from the introduction).  

\begin{corollary}
\label{cor:highest}
$
\displaystyle s_{d_k}(k)= \frac{k}{p+1} + O(\log(k)).
$
\end{corollary}

\begin{proof}
Note that $\displaystyle d_k = {k\mu_0(N)\over 12} + O(1)$ and take $i=d_k$ in Theorem \ref{thm:asymptotic}.
\end{proof}

Recall that $d_{k,p} := \dim S_k(\Gamma_0(Np))$. Then, consider the set
$$
\mathbf x_k = \left\{ \frac{s_i(k)}{k-1} \st 1\leq i \leq d_{k,p}
\right\}  \subseteq [0,\infty).
$$
Let $\mu^{(p)}_k$ be the probability measure on $[0,\infty)$ uniformly supported on $\mathbf x_k$. We refer to \cite[Sections 1.1--1.2]{Serre-Equidistribution} for the notion of weak convergence and its relationship to equidistribution.

\begin{corollary}\label{corollary:gouvea-distribution}
As $k\goto \infty$, the measures $\mu^{(p)}_k$ weakly converge to a probability measure $\mu^{(p)}$ on $[0,1]$ which is supported on  $[0,\frac{1}{p+1}]  \cup \{\frac{1}{2}\} \cup [\frac{p}{p+1},1]$.  Explicitly, $\mu^{(p)}(\set{1\over 2}) = {p-1\over p+1}$ and the remaining mass is uniformly distributed over $[0,\frac{1}{p+1}]  \cup [\frac{p}{p+1},1]$.  
\end{corollary}

\begin{proof}
This is clear from Theorem \ref{thm:asymptotic} and the asymptotics for $d_k, d_k^{\new}$ and $d_{k,p}$ (for example, see the proof of Proposition \ref{prop:approximate-zero-location} for $d_k$ and $d_k^{\new}$; an asymptotic for $d_{k,p}$ is easily obtained from those two).
\end{proof}

\begin{remark}
The measures $\mu_k^{(p)}$ clearly depend on $N$, even if $N$ is suppressed from our notation. However, it is interesting that the limit $\mu^{(p)}$ does not.
\end{remark}

\begin{remark}
The key point in the proof of Theorem \ref{thm:asymptotic}  is the analysis in Proposition \ref{prop:ith_newtonslope} below. If one could prove an analog of Proposition \ref{prop:ith_newtonslope} for the Fredholm series of $U_p$ then the proof of Theorem \ref{thm:asymptotic} would go through as written.
\end{remark}

The rest of this section is devoted to proving Theorem \ref{thm:asymptotic}.  Our strategy is to prove an analoge of Theorem \ref{thm:asymptotic} for ghost $\Delta$-slopes first and, from this, make conclusions about ghost slopes. We will need two short lemmas.

\begin{lemma}\label{lemma:valuations}
Suppose that $y,\lambda > 0$ are integers and $p$ is a prime number. Then
\begin{equation*}
v_p(\lambda!) \leq \sum_{i=0}^{\lambda-1} v_p(y+i) \leq v_p((\lambda-1)!) + \floor{\log_p(y+\lambda)} + \min(v_p(y),v_p(\lambda)).
\end{equation*}
\end{lemma}
\begin{proof}
Write 
\begin{equation*}
s(y) = \sum_{i=0}^{\lambda-1} v_p(y+i) = v_p\left({y + \lambda -1 \choose \lambda}\lambda!\right).
\end{equation*}
Since binomial coefficients are integers we immediately get the lower bound $v_p(\lambda!) \leq s(y)$. On the other hand, we can also write
\begin{equation*}
{y+\lambda-1\choose \lambda}\lambda! = {(y+\lambda-1)!\over (y-1)!(\lambda-1)!}(\lambda-1)!,
\end{equation*}
so for the upper bound it suffices to see 
\begin{equation}\label{eqn:to-show-valuation-lemma}
v_p\left( {(y+\lambda-1)!\over (y-1)!(\lambda-1)!} \right) \leq \floor{\log_p(y+\lambda)} + \min(v_p(y),v_p(\lambda)).
\end{equation}
Since \eqref{eqn:to-show-valuation-lemma} is symmetric in $\lambda$ and $y$, we may assume that $v_p(y) \leq v_p(\lambda)$. In that case, the classical estimate $v_p\left({n\choose k}\right) \leq \floor{\log_p(n+1)}$ yields
\begin{equation*}
v_p\left( {(y+\lambda-1)!\over (y-1)!(\lambda-1)!} \right) = v_p(y) + v_p\left({y+\lambda-1 \choose y}\right) \leq v_p(y) + \floor{\log_p(y + \lambda)}.
\end{equation*}
This completes the proof.
\end{proof}

Now set $\delta$ be the size of the torsion subgroup in $\Z_p^\times$. Thus $\delta = p-1$ if $p$ is odd and $\delta = 2$ if $p = 2$. If $k_0 \in \Z$, $\lambda > 0$ and $p$ is a prime then we define
\begin{equation*}
P_{k_0,\lambda}(w) = (w-w_{k_0})(w-w_{k_0-\delta})\dotsb (w-w_{k_0 - (\lambda-1)\delta}).
\end{equation*}
Thus $P_{k_0,\lambda} \in \Z[w]$ has $\lambda$-many zeros, the highest zero is $k_0$, and the zeros are an arithmetic progression of difference $p-1$ if $p$ is odd and $2$ if $p$ is even (compare with Proposition \ref{prop:change_in_zeroes}). Write $q=p$ if $p$ is odd and $q = 4$ if $p=2$
\begin{lemma}\label{lemma:john-fixed-up}
Assume that $k \congruent k_0 \bmod \delta$ and $P_{k_0,\lambda}(w_{k}) \neq 0$. Then
\begin{equation*}
v_p(P_{k_0,\lambda}(w_{k})) = {q \lambda \over p-1} + O(\log \lambda,\log\abs{k-k_0}))
\end{equation*}
\end{lemma}
\begin{proof}
For any $k,k'$ we have $v_p(w_{k} - w_{k'}) = v_p(2p) + v_p(k-k')$. Since $P_{k_0,\lambda}(w_{k}) \neq 0$ and $k \congruent k_0 \bmod \delta$ we have either $k < k_0-(\lambda-1)\delta$ or $k_0 < k$. Thus we deduce that
\begin{equation}\label{eqn:first-eqn}
v_p(P_{k_0,\lambda}(w_{k})) = v_p(2p)\lambda + \sum_{i=0}^{\lambda-1} v_p(x + i\delta)
\end{equation}
where $x = k-k_0$ or $x=k_0 - (\lambda-1)\delta - k$ depending on which choice makes $x > 0$. Note that $x \congruent 0 \bmod \delta$. So, replacing $x$ by $y = x/\delta$,  \eqref{eqn:first-eqn} becomes
\begin{equation}\label{eqn:second}
v_p(P_{k_0,\lambda}(w_{k})) = \vartheta\lambda + \sum_{i=0}^{\lambda-1} v_p(y+i)
\end{equation}
where $\vartheta = 1$ if $p$ is odd and $\vartheta = 3$ otherwise. 

By \eqref{eqn:second}
 and Lemma \ref{lemma:valuations} we see 
\begin{equation}\label{eqn:kanye-west}
\vartheta\lambda + v_p(\lambda!) \leq v_p(P_{k_0,\lambda}(w_{k}))\\ \leq \vartheta\lambda + v_p((\lambda-1)!) + \floor{\log_p(y+\lambda)} + \min(v_p(y),v_p(\lambda)).
\end{equation}
On the left-hand side of \eqref{eqn:kanye-west} we have
\begin{equation*}
\vartheta\lambda + v_p(\lambda!) \geq {(\vartheta(p-1) + 1)\lambda \over p-1} - \ceil{\log_p(\lambda)} = {q\lambda \over p-1} - \ceil{\log_p(\lambda)},
\end{equation*}
and on the right-hand side \eqref{eqn:kanye-west} we have
\begin{equation*}
\vartheta\lambda + v_p((\lambda-1)!) \leq \vartheta\lambda + {\lambda - 1 \over p-1} \leq {q \lambda \over p-1}.
\end{equation*}
(Here we've used the classical formula of Legendre for $v_p(n!)$.) By \eqref{eqn:kanye-west} we get
\begin{equation*}
-\ceil{\log_p(\lambda)} \leq v_p(P_{k_0,\lambda}(w_{k})) - {q\lambda \over p-1} \leq \floor{\log_p(y+\lambda)} + \min(v_p(y),v_p(\lambda)).
\end{equation*}
Since $y = \abs{k-k_0} + O(\lambda)$ (and $\log_p x = O(\log x)$), we're finished.
\end{proof}

Now fix $\cal W_{\ve}$ and write $G(w,t) = 1 + \sum g_i(w)t^i$ for the ghost series over $\cal W_{\varepsilon}$. We assume all weights $k$ are in $\cal W_{\varepsilon}$ in what follows. Recall that $\Delta_i = g_i/g_{i-1}$, and if $\Delta_i(w_{k})$ is well-defined then $v_p(\Delta_i(w_{k}))$  is the $i$-th $\Delta$-slope in weight $k$.  We define
$$
\Delta_i^*(w_{k}) := \begin{cases}
(w-w_{k}) \Delta_i(w_{k})  & \text{if~}\Delta_i \text{~has~a~pole~at~}w_{k},\\
\displaystyle \frac{\Delta_i(w_{k})}{w-w_{k}}   & \text{if~}\Delta_i \text{~has~a~zero~at~}w_{k},\\
\Delta_i(w_{k}) & \text{otherwise}.
\end{cases}
$$
Since $\Delta_i$ only has simple zeros or poles, $\Delta_i^*$ has no zeros or poles.

\begin{proposition}\label{prop:ith_newtonslope}
We have
\begin{equation*}
v_p\left(\Delta_i^{\ast}(w_{k})\right) = {12 i \over \mu_0(N)(p+1)} + O( \log k, \log i)
\end{equation*}
\end{proposition}
\begin{proof}
Recall our standard practice of writing $\Delta_i = \Delta_i^+/\Delta_i^-$ in lowest terms. Write $\lambda_i^+$ for the number of  zeros of $\Delta_i^+$ and $\lambda_i^-$ for the number of zeros of $\Delta_i^-$. Write $k_i^+ = \HZ(\Delta_i^+)$ and $k_i^- = \HZ(\Delta_i^-)$. We note it suffices to prove the result separately for pairs $(i,k)$ ranging over a finite number of disjoint domains. With this in mind, we will focus only on the pairs $(i,k)$ such that $w_{k}$ is a zero of $\Delta_i^+$ and leave the other possible pairs for the reader. We will also assume that $p > 3$ or $N > 1$ for simplicity.\footnote{The proof below can easily be modified to handle $p=2,3$ and $N=1$. For example, the corrected formula for $\Delta_i^+$ is $\Delta_i^+(w_{k}) = P_{k_1^+-4,\lambda^+-1}(w_{k})\cdot(w_{12i+2}-w_{k})$, and so the estimates that follow will only be off by $O(\log k, \log i)$.\label{footnote:p=23N=1-asy-fix}}

By \eqref{eqn:Deltaizeros}, if $\Delta_i^+(w_{k}) = 0$ then $k = O(i)$ and thus our goal is to show that 
\begin{equation*}
v_p(\Delta_i^{\ast}(w_{k})) = {12 i \over \mu_0(N)(p+1)} + O(\log i).
\end{equation*}
By Proposition \ref{prop:change_in_zeroes}(a), using that either $p > 3$ or $N > 1$, we have
\begin{equation*}
\Delta_i^+(w_{k}) = P_{k_i^+,\lambda'}(w_{k})\cdot (w-w_{k}) \cdot P_{k - \delta,\lambda''}(w_{k})
\end{equation*}
where $\lambda' + \lambda'' = \lambda^+ - 1$. So, by definition of $\Delta_i^{\ast}$ we have
\begin{equation*}
v_p\left(\Delta_i^{\ast}(w_{k})\right) = v_p\left(P_{k_i^+,\lambda'}(w_{k})\right) + v_p\left(P_{k - \delta,\lambda''}(w_{k})\right) - v_p\left(P_{k_i^-,\lambda^-}(w_{k})\right)
\end{equation*}
and $k \congruent k_i^{\pm} \bmod \delta$. Next, by Proposition \ref{prop:change_in_zeroes}(b,c) we have $k_i^+ = O(i)$ and $k_i^- = O(i)$; by Lemma \ref{lemma:zeropole_count} we have $\lambda^+ = O(i)$ and $\lambda^- = O(i)$. Since $k = O(i)$ as well, Lemma \ref{lemma:john-fixed-up} implies
\begin{equation*}
v_p(\Delta_i^{\ast}(w_{k})) = {q\over p-1}(\lambda ' + \lambda'' - \lambda^-) + O(\log i) = {q\over p-1}(\lambda^+ - \lambda^-) + O(\log i).
\end{equation*}
Finally by Lemma \ref{lemma:zeropole_count} we have
\begin{equation*}
{q \over p-1}(\lambda^+ - \lambda^-) = {12i \over \mu_0(N)(p+1)} + O(1).
\end{equation*}
This completes the proof.
\end{proof}
To pass from asymptotic control of ghost $\Delta$-slopes as in Proposition \ref{prop:ith_newtonslope} to asymptotic control of ghost  slopes, we need to show that $i=d_k$ and $i=d_k+d_k^{\new}$ are asymptotically indices of points on $\NP(G_k)$ (Lemma \ref{lemma:break} below). First, we give asymptotic control of the ghost slopes over ``oldform'' and ``newform'' ranges.

\begin{lemma}
\label{lemma:Gbounds}
If $x > 0$ then there exists a $k'$ such that if $k \geq k'$ then
\begin{enumerate}
\item $\displaystyle v_p(\Delta_i(w_{k})) \displaystyle < \left(\frac{1}{p+1} + x\right) k$ for all $i\leq d_k$, and
\item $\displaystyle v_p(\Delta_{i}(w_{k})) \displaystyle > \left(\frac{p}{p+1} - x\right)k$ for all $i \geq d_k + d_k^{\new}$.
\end{enumerate}
\end{lemma}

\begin{proof}
We check the claim (a) of the lemma as (b) is handled similarly. Note that if $i\leq d_k$ then $\Delta_i^{\ast}(w_{k}) = \Delta_i(w_{k})$, and $i = O(k)$. So, Proposition \ref{prop:ith_newtonslope} implies there is a constant $A\geq0$ such that if $i\leq d_k$ then
$$
v_p(\Delta_i(w_{k})) \leq \frac{12i}{\mu_0(N)(p+1)} + A \log k.
$$
Since $i\leq \displaystyle d_k = \frac{k\mu_0(N)}{12} + O(1)$, 
$$
v_p(\Delta_i(k)) \leq
\frac{k}{p+1}  + A \log k + B
$$
for some $B>0$. The lemma clearly follows now.
\end{proof}

\begin{lemma}
\label{lemma:ss_slope}
Set $y_i(k) = v_p(g_i(w_{k}))$. Then,
$$
\frac{y_{d_k+d_k^{\new}}(k) - y_{d_k}(k)}{d_k^{\new}} = \frac{k}{2} + O(\log k).
$$
\end{lemma} 

\begin{proof}
By Proposition \ref{prop:ith_newtonslope}, we have
\begin{align}
y_{d_k}(k) \label{eqn:ydk-shorter}
&= \sum_{i=1}^{d_k} v_p(\Delta_i (w_{k})) 
= \sum_{i=1}^{d_k} \frac{12i}{\mu_0(N)(p+1)} + O(\log(k)) \\
&= \frac{12}{\mu_0(N)(p+1)} \binom{d_k}{2} + O(k\log k) 
= \frac{6(d_k)^2 }{\mu_0(N)(p+1)} + O(k\log k) \nonumber
\end{align}
Among $d_k < i < d_k + d_k^{\new}$, $w_{k}$ is a zero of $\Delta_i$ exactly as many times as it is a pole (by construction), and so
$$
\prod_{j=d_k+1}^{d_k+d_k^{\new}} \Delta_i(w_{k}) = \prod_{j=d_k+1}^{d_k+d_k^{\new}} \Delta_i^*(w_{k}).
$$
Arguing as above, using Proposition \ref{prop:ith_newtonslope}, gives
\begin{equation}\label{eqn:ydk-longer}
y_{d_k+d_k^{\new}}(k) = \sum_{i=1}^{d_k+d_k^{\new}} v_p(\Delta_i^*(w_{k}))
  =  {6(d_k+d_k^{\new})^2 \over \mu_0(N)(p+1)}    + O(k\log k) 
\end{equation}
Combining \eqref{eqn:ydk-shorter} and \eqref{eqn:ydk-longer}, we deduce that
\begin{align*}
\frac{y_{d_k+d_k^{\new}(k)} - y_{d_k}(k)}{d_k^{\new}} 
&= {1 \over d_k^{\new}}\cdot\left(\frac{6((d_k+d_k^{\new})^2-(d_k)^2)}{\mu_0(N)(p+1)} +O(k\log k)\right)  \\
&= \frac{6d_{k,p}}{\mu_0(N)(p+1)}  + O(\log k) 
= \frac{k}{2} + O(\log k),
\end{align*}
as desired.
\end{proof}

\begin{lemma}
\label{lemma:break}
For $k\gg0$, $i=d_k$ and $i=d_k+d_k^{\new}$ are indices of break points on $\NP(G_k)$.
\end{lemma}
\begin{proof}
This is immediate from the two previous lemmas, and the next lemma whose proof we leave to the reader.
\end{proof}
\begin{lemma}
\label{lemma:three_part}
Consider a collection $\cP = \{(i,y_i) : i\geq 0\}$ such that $y_i \in \R_{\geq 0} \union \set{\infty}$ and $y_i = \infty$ if and only if $N_1<i<N_2$ for some $N_i \geq 0$.  If $i < j$, set $\Delta_{i,j} = {y_j-y_i\over j-i}$, and set $\Delta_i := \Delta_{i-1,i}$.  Assume that there are constants $\gamma_i$ such that:
\begin{enumerate}
\item 
If $i \leq N_1$ then $\Delta_i \leq \gamma_1$; 
\item 
If $N_2 < i$ then $\Delta_i \geq \gamma_2$; and 
\item
$\gamma_1 < \Delta_{N_1,N_2} < \gamma_2$.
\end{enumerate}
Then, $N_1$ and $N_2$ are indices of break points of $NP(\cP)$.
\end{lemma}

We now prove the main theorem of this section.

\begin{proof}[Proof of Theorem \ref{thm:asymptotic}]
Let $C = {12 \over \mu_0(N)(p+1)}$. We need to show that there exists a constant $A > 0$ such that 
\begin{enumerate}[(i)]
\item if $d_k < i \leq d_k + d_k^{\new}$ then $-A \log k \leq s_i(k) - k/2 \leq A \log k$, and
\item If $i \leq d_k$ or $i > d_k + d_k^{\new}$ then $-A\max(\log k, \log i) \leq s_i(k) - Ci \leq A \max( \log k, \log i)$
\end{enumerate}
It suffices to check $A$ exists independently for each of the four bounds.

For (i), if $k$ is fixed then only finitely many $i$ satisfy $i \leq d_k + d_k^{\new}$ and so we may, without loss of generality, assume that $k$ is sufficiently large. In that case, Lemma \ref{lemma:break} implies that the indices $d_k$ and $d_k+d_k^{\new}$ are indices of break points on $\NP(G_k)$, and so Lemma \ref{lemma:ss_slope} proves (i) holds for some $A>0$.

For case (ii), we write $y_i(k) = v_p(g_i(w_{k}))$. To compute an asymptotic for $s_i(k)$ it suffices to assume that $(i-1,y_{i-1}(k))$ is a break point of the Newton polygon. In that case, by definition of Newton polygon, we know that $s_i(k) \leq v_p(\Delta_i(w_{k}))$ for all such $i$ and all $k$ and so Proposition \ref{prop:ith_newtonslope} gives us upper bounds for (ii).

Now we deal with lower bounds. We may separately assume that $i \leq d_k$ and $i > d_k + d_k^{\new}$. First assume that $i \leq d_k$. Then $i = O(k)$, so we can choose a constant $A > 0$ such that if $m \leq d_k$ then $v_p(\Delta_m(w_{k})) \geq Cm - A \log k$ (Proposition \ref{prop:ith_newtonslope}). In particular, if $j\geq 0$ and $i + j \leq d_k$ then
\begin{multline}\label{eqn:ileqdk}
y_{i+j}(k) = y_{i-1}(k) + \sum_{m=i}^{i+j} v_p(\Delta_m(w_{k})) \geq y_{i-1}(k) + \sum_{m=i}^{i+j} (Cm - A \log k)\\
\geq y_{i-1}(k) + (Ci + {Cj\over 2})(j+1) - (A \log k)(j+1).
\end{multline}
Thus, 
\begin{equation}\label{eqn:ileqdk2}
{y_{i+j}(k) - y_{i-1}(k)\over j+1} \geq Ci + {Cj\over 2} - A \log k.
\end{equation}
The right-hand side of \eqref{eqn:ileqdk2} is minimized at $j = 0$ and so we deduce 
\begin{equation}\label{eqn:nearlyileqdkslope}
{y_{i+j}(k) - y_{i-1}(k) \over j+1} \geq Ci - A \log k
\end{equation}
for all $i \leq d_k$ and $j \geq 0$ with $i + j \leq d_k$. Finally by Lemma \ref{lemma:break}, except for finitely many $k$, and thus finitely many $i \leq d_k$, $s_i(k)$ is the slope of line segment connecting index $i-1$ to index $i+j$ for some $i + j \leq d_k$. Thus we conclude $s_i(k) - Ci \geq - A \log k$ for some $A > 0$ and all $i \leq d_k$.

Now consider the case where $i > d_k + d_k^{\new}$. If $i$ is fixed then $i \leq d_k$ except for finitely many $k$ and so we may also suppose in what follows that $i$ is sufficiently large (to be determined). Continuing, since $i > d_k + d_k^{\new}$, we have $k = O(i)$ and so the analog of \eqref{eqn:ileqdk2} is 
\begin{equation}\label{eqn:igreaterdk}
{y_{i+j} - y_{i-1} \over j+1} \geq Ci + {Cj\over 2} - A \log(i+j)
\end{equation}
for all $j\geq 0$. The right-hand side of \eqref{eqn:igreaterdk}, as a function of $j$, has a unique local minimum at $j = 2A/C - i$ and so if we suppose that $i > 2A/C$ then the right-hand side of \eqref{eqn:igreaterdk} is minimized at $j=0$ on the domain $j \geq 0$. The proof is now completed  just as before.
\end{proof}

\section{Halos and arithmetic progressions}
\label{sec:global}

The goal of this section is to prove that for weights $\kappa$ with $w_{\kappa} \nin \Z_p$, the slopes of $\NP(G_{\kappa})$ are, except for a finite number of terms, a finite union of arithmetic progressions whose common difference can be explicitly determined. Throughout we will {{\em assume that $p$ is odd}}. See Remark \ref{rmk:p=2-aps-remark} for $p=2$.

Fix a component $\cal W_{\ve}$  of $p$-adic weight space, and we implicitly assume all weights lie within $\cal W_{\ve}$ in what follows. Set $\overline{G(w,t)} = \overline{G^{(\varepsilon)}(w,t)} \in \F_p[[w,t]]$ for the reduction modulo $p$ of the ghost series. We write $\NP(\bar{G})$ for the Newton polygon of $\overline{G(w,t)}$ computed with respect $w$-adic valuation on the coefficients in $\F_p[[w]]$. Write
\begin{equation*}
C_{p,N} := {p(p-1)(p+1)\mu_0(N)\over 24},
\end{equation*}
and if $r \geq 0$ is an integer write $C_{p,N,r} = p^r C_{p,N}$. Since $p$ is odd, $p(p-1)(p+1) \congruent 0 \bmod 24$, so $C_{p,N}$ is an integer divisible by $\mu_0(N)$.  Recall that if $w_{\kappa} \nin \Z_p$ then we write $\alpha_{\kappa} = \sup_{w' \in \Z_p} v_p(w_{\kappa}-w') \in (0,\infty)$.

\begin{theorem}\label{theorem:global-halo-progressions}
Assume that $w_{\kappa} \nin \Z_p$ and write $r = \floor{\alpha_\kappa}$. Then,
the slopes of $\NP(G_{\kappa})$  form a finite union of $C_{p,N,r}$-many arithmetic progressions with common difference
\begin{equation*}
{(p-1)^2\over 2}\left(\alpha_{\kappa} + \sum_{v=1}^r (p-1)p^{r-v}\cdot v \right)
\end{equation*}
up to finitely many exceptional slopes contained within the first $C_{p,N,r}$ slopes.
\end{theorem}
In Theorem \ref{theorem:global-halo-progressions}, the condition that $r = 0$ is equivalent to $0 < v_p(w_{\kappa}) < 1$, and in that case $\alpha_{\kappa} = v_p(w_{\kappa})$. The conclusion is the the slopes of $\NP(G_{\kappa})$ are, up to a finite number of exceptions, a finite union of $C_{p,N}$-many arithmetic progressions of common difference ${(p-1)^2\over 2}\cdot v_p(w_{\kappa})$. From the ghost spectral halo (Theorem \ref{thm:gsh} in the introduction) we deduce:

\begin{corollary}
\label{cor:progs}
The slopes of $\NP(\overline{G})$ are a finite union of $C_{p,N}$-many arithmetic progressions whose common difference is $\frac{(p-1)^2}{2}$ up to finitely many exceptional slopes contained within the first $C_{p,N}$ slopes.
\end{corollary}

\begin{remark}
The exceptional slopes in $\NP(\bar G)$ should not exist (see the comments after Theorem \ref{theorem:intro-aps-boundary} in the introduction) but we have not pursued proving this stronger statement.
\end{remark}

The remainder of the section is devoted to proving Theorem \ref{theorem:global-halo-progressions}.  Our method, as in Section \ref{subsec:distribution}, is to first verify a corresponding statement for ghost $\Delta$-slopes, and, from this,  deduce our result about ghost slopes.   To this end, here is a general lemma on Newton polygons.

\begin{lemma}
\label{lemma:newton_slopes}
Consider a collection $\cP = \{(i,y_i) : i\geq 0\}$ such that $y_i \in \R_{>0}$.  If the $\Delta$-slopes of $\cP$ form a union of $C$ arithmetic progressions with common difference $\delta$, then the same holds for the slopes of $\NP(\cP)$ up to finitely many exceptional slopes contained within the first $C$ slopes.
\end{lemma}

\begin{proof}
This follows immediately from observing that if $x\geq C$ is the index of a breakpoint of $\NP(\cP)$, then $x-C$ is also the index of a breakpoint of $\NP(\cP)$.
\end{proof}

To deduce Theorem \ref{theorem:global-halo-progressions} from Lemma \ref{lemma:newton_slopes}, we need to verify that the Newton slopes of $G_\kappa$ are a finite union of arithmetic progressions. This is not quite true, but we will show it is true after excluding the weight $w = w_{2}$ from ever appearing as a zero of a ghost coefficient. This modification has no effect on $\NP(G_\kappa)$. Specifically:

\begin{lemma}\label{lemma:remove-wt2-zero}
 For each $i\geq 1$, write $g_{i}^{\sharp}(w) = g_i(w)(w-w_{2})^{-m}$ where $m = \ord_{w = w_{2}} g_i(w)$. Set $G^{\sharp}(w,t) = 1 + \sum g_i^{\sharp}(w)t^i$ and $\bar{G^{\sharp}}$ as its reduction modulo $p$. Then, $\NP(G_\kappa) =\NP(G_{\kappa}^{\sharp})$ for all $\kappa \in \cal W$, and $\NP(\bar G) = \NP(\bar{G^{\sharp}})$.
\end{lemma}

\begin{proof}
This follows from (the equality in) Lemma \ref{lemma:explicit-ghost-ordinary-dims}(b).
\end{proof}
\begin{framed}
\vspace{-.20cm}
\textbf{Convention:} for the rest of this section {\em we replace the $g_i(w)$ by $g_i^{\sharp}(w)$}.
\vspace{-.20cm}
\end{framed}
 
We now aim to show that the $\Delta$-slopes of (the newly defined) $G_\kappa$ form a finite union of arithmetic progressions.  Recall, $\Delta_i = g_i/g_{i-1}$ and $\Delta_i = \Delta_i^+/\Delta_i^-$ with $\Delta_i^{\pm} \in \Z_p[[w]]$ and $\gcd(\Delta_i^+,\Delta_i^-) = 1$. As preparation, we will compare the zeros of $\Delta_i^{\pm}$ to those of $\Delta^{\pm}_{i+C_{p,N,r}}$. We write $\HZ(-)$ and $\LZ(-)$ for the highest and lowest zeros as in Section \ref{section:explicit-analysis} (if they exist).

\begin{lemma}
\label{lemma:lambdazp_change}
\leavevmode
\begin{enumerate}
\item If $\lambda(\Delta_i^+)>0$, then 
\begin{equation*}
\HZ(\Delta^+_{i+C_{p,N}}) =  \HZ(\Delta_i^+) + \frac{p(p+1)(p-1)}{2}
~\text{~~~and~~~}~
\LZ(\Delta^+_{i+C_{p,N}}) =  \LZ(\Delta^+_i) + p(p-1).
\end{equation*}
\item If $\lambda(\Delta_i^-)>0$, then 
$$
\HZ(\Delta^-_{i+C_{p,N}}) =  \HZ(\Delta^-_i) + p(p-1)
~\text{~~~and~~~}~
\LZ(\Delta^-_{i+C_{p,N}}) =  \LZ(\Delta^-_i) + \frac{(p+1)(p-1)}{2}.
$$
\end{enumerate}
\end{lemma}
\begin{proof}
We prove the assertions for $\Delta_i^+$ and leave part (b) for the reader (the proofs are analogous).

We recall that by \eqref{eqn:Deltaizeros}, for each $i$, $\HZ(\Delta_i^+)$ is the largest $k \in \cal W_{\varepsilon}$ such that $d_k < i$ (with $k\geq 4$, convention in this section). Next, write $C = {p(p+1)(p-1)\over 2} = {12 C_{p,N}\over \mu_0(N)} \congruent 0 \bmod p-1$. Thus $k \mapsto k+C$ preserves the component of weight space. Moreover, since $d_{k+12} = d_k + \mu_0(N)$ we have $d_{k + C} = d_k + C_{p,N}$.

Now let $k = \HZ(\Delta_i^+)$ and $k' = \HZ(\Delta_{i+C_{p,N}}^+)$. The previous paragraph implies that $k + C \leq k'$. Write $k'-C = k + j(p-1)$ for some $j\geq 0$. If $j > 0$ then $k'-C > k$ and so by definition of highest zero, $i \leq d_{k'-C}$. But the previous paragraph then implies that $i + C_{p,N} \leq d_{k'}$, which is a contradiction to the definition of $k'$. 

Proving the formula for $\LZ(\Delta_{i+C_{p,N}}^+)$ is slightly more tedious. Set $k = \LZ(\Delta_i^+)$  and $k' = \LZ(\Delta_{i+C_{p,N}}^+)$. Then $k' \leq k + p(p-1)$ because Lemma \ref{lemma:robs-tedious-calculation} implies that
\begin{equation}\label{eqn:lemmaa6-equation}
d_{k+p(p-1)} + \left\lfloor \frac{d^{\new}_{k+p(p-1)}}{2} \right\rfloor = d_k + \left\lfloor \frac{d^{\new}_k}{2} \right\rfloor  + C_{p,N}.
\end{equation}
If $p=3$ and $N=1$ then $k' = k+6$ by Table \ref{table:least-zeros}. Thus we assume that $p>3$ or $N>1$. In particular, Lemma \ref{lemma:dim-formulas-increasing} then implies that since we already showed that $k' \leq k + p(p-1)$ we may finish by showing
\begin{equation}\label{eqn:newest-realization}
d_{k+(p-1)(p-1)} + \bfloor{{d_{k+(p-1)(p-1)}^{\new}\over 2}} < i + C_{p,N}.
\end{equation}
By definition of $k = \LZ(\Delta_i^+)$, $i \leq d_k +\floor{{d_k^{\new}\over 2}}$ and either
\begin{enumerate}[i)]
\item $k - (p-1) < 4$, or
\item $k-(p-1) \geq 4$ but $d_{k-(p-1)} + \floor{{d_{k-(p-1)}^{\new}\over 2}} < i$.
\end{enumerate}
If (ii) holds then \eqref{eqn:newest-realization} is immediate from Lemma \ref{lemma:robs-tedious-calculation} and the assumption in (ii).

It remains to handle case (i):  $4 \leq k \leq p+1$. Then, $k$ is the lowest integer weight $k\geq 4$ on our fixed component and so the assumption that $\lambda(\Delta_i^+) > 0$ and $k = \LZ(\Delta_i^+)$ implies that
\begin{equation}\label{eqn:gold-star}
d_k < i \leq d_k + \left\lfloor{{d_k^{\new}\over 2}}\right\rfloor.
\end{equation}
First assume $p\neq 3$. Then Lemma \ref{lemma:robs-tedious-calculation} reduces \eqref{eqn:newest-realization} to showing 
\begin{equation*}
d_k + \bfloor{{d_k^{\new}\over 2}} < i + {(p-1)(p+1)\over 24}\mu_0(N)
\end{equation*}
instead. But by \eqref{eqn:gold-star}, this reduces to checking for $i=d_k+1$, and in that case checking
\begin{equation*}
\bfloor{d_k^{\new}\over 2} < 1 + {(p-1)(p+1)\over 24}\mu_0(N).
\end{equation*}
We leave this final point for the reader.

Now assume that $p=3$, so that our assumption now is that $k=4$ is the lowest zero of $\Delta_i^+$. We have $C_{3,N} = \mu_0(N)$. By \eqref{eqn:gold-star} it is enough to show \eqref{eqn:newest-realization} when $i = d_4+1$, and thus we need to check
\begin{equation*}
d_8 + \bfloor{{d_8^{\new}\over 2}} < d_4 + 1 + \mu_0(N),
\end{equation*}
which we also leave for the reader.
\end{proof}

\begin{proposition}
\label{prop:lambda_change}
If $i\geq 1$ then
\begin{enumerate}
\item $\lambda(\Delta^+_{i + C_{p,N}}) = \lambda(\Delta^+_i) + \frac{p(p-1)}{2}$, and
\item $\lambda(\Delta^-_{i + C_{p,N}}) = \lambda(\Delta_i^-) + \frac{p-1}{2}$.
\end{enumerate}
\end{proposition}
\begin{proof}
If $p=3$ and $N=1$, then this proposition follows from Table \ref{table:least-zeros}. Otherwise,
for each $i$, Proposition \ref{prop:change_in_zeroes}(a) (valid by our exclusion of $p=3$ and $N=1$) implies that
\begin{equation}\label{eqn:lambda-difference}
\lambda(\Delta_i^{\pm}) = 1 + {1\over p-1}(\HZ(\Delta_i^{\pm}) - \LZ(\Delta_i^{\pm})).
\end{equation}
If $\lambda(\Delta_i^+) > 0$ then (a) follows \eqref{eqn:lambda-difference} and Lemma \ref{lemma:lambdazp_change}(a), and if $\lambda(\Delta_i^-) > 0$ then (b) follows from \eqref{eqn:lambda-difference} and Lemma \ref{lemma:lambdazp_change}(b).

If $\lambda(\Delta_i^+) = 0$ we proceed as follows (and leave the reader to deal with $\lambda(\Delta_i^-) = 0$). First, if $k < 4 $ is even then we re-define $d_k$ and $d_k^{\new}$ using the formulas \eqref{eqn:dim-formula-old} and \eqref{eqn:dim-formula-new} in Appendix \ref{app:dimension-formulae}. We then define $\HZ(\Delta_i^+)$ and $\LZ(\Delta_i^+)$ by insisting that \eqref{eqn:Deltaizeros} holds, i.e.\ $\HZ(\Delta_i^+)$ is the largest $k \in \cal W_{\varepsilon}$ such that $d_k < i$ and $\LZ(\Delta_i^+)$ is the least $k \in \cal W_{\varepsilon}$ such that $i\leq d_k + \bfloor{{d_k^{\new}/ 2}}$.

Continue to suppose that $\lambda(\Delta_i^+) = 0$. Then we must have $\HZ(\Delta_i^+) < \LZ(\Delta_i^+)$ (since otherwise $\Delta_i^+$ would have a zero). Moreover, by definition, $\HZ(\Delta_i^+) \congruent \LZ(\Delta_i^+) \bmod p-1$. We claim that $\LZ(\Delta_i^+) - \HZ(\Delta_i^+) = p-1$. 

We will prove our claim by contradiction. First, we observe that $\HZ(\Delta_i^+) \geq 0$. Indeed, if $N > 1$ then it is easy to see that if $k\leq 0$ then $d_k < 0 \leq i$, and if $N=1$ then $d_2 = -1 < i$ for any $i\geq 0$, and thus $\HZ(\Delta_i^+) \geq 2$ in fact. Next, if $\LZ(\Delta_i^+) - \HZ(\Delta_i^+) > p-1$ then since $\HZ(\Delta_i^+) \congruent \LZ(\Delta_i^+) \bmod p-1$ (by definition) we can find a $k \in \cal W_{\varepsilon}$ such that $\HZ(\Delta_i^+) < k < \LZ(\Delta_i^+)$. This implies $i \leq d_k$ and $d_k + \floor{{d_k^{\new}\over 2}} < i$, whence $d_k^{\new} < 0$. But this implies that $k \leq 0$, which is a contradiction.\footnote{If $k\geq 4$ then obviously $d_k^{\new}\geq 0$ and the reader may check that $d_2^{\new} \geq 0$ for any $p$ and $N$, given our overwritten definition.}

Finally, the reader may check that the proof of Lemma \ref{lemma:lambdazp_change} extends to the new definitions of $\HZ(\Delta_i^+)$ and $\LZ(\Delta_i^+)$, and thus 
\begin{equation*}
\lambda(\Delta^+_{i+C_{p,N}}) = {\HZ(\Delta_i^+)-\LZ(\Delta_i^+)\over p-1} + 1 + {p(p-1)\over 2} = {p(p-1)\over 2}.
\end{equation*}
This completes the proof.
\end{proof}

\begin{remark}
The $i$-th $w$-adic $\Delta$-slope of $\bar G$ is $\lambda(g_i) - \lambda(g_{i-1}) = \lambda(\Delta_i^+) - \lambda(\Delta_i^-)$.  Thus, 
Proposition \ref{prop:lambda_change} together with Lemma \ref{lemma:newton_slopes} implies Corollary \ref{cor:progs}.
\end{remark}

We briefly unwind the condition $w_{\kappa} \nin \Z_p$. We thank Erick Knight for pointing out the equivalence in  Lemma \ref{lemma:equivalent-conditons-inZp} below. Write $\Z_p^{\nr}$ for the ring of integers in the maximal unramified extension of $\Q_p$ contained in $\bar \Q_p$, and $\omega: \bar \F_p^\times \goto ( \Z_p^{\nr})^\times$ for the Teichm\"uller lift.
\begin{lemma}\label{lemma:Zpnr}
If $x_0 \in \cal O_{\C_p}$ and there exists $x' \in \Z_p^{\nr} - \Z_p$ such that $v_p(x_0 - x') > v_p(x')$ then $v_p(x_0 - x) = \min(v_p(x_0),v_p(x))$ for all $x \in \Z_p$.
\end{lemma}
\begin{proof} Suppose $x \in \Z_p$ and $v_p(x_0) = v_p(x)$. We will show $v_p(x_0 - x) = v_p(x_0)$. If $x_0 \in  \Z_p^{\nr} - \Z_p$ then the result is clear. Indeed,  the assumption $x_0 \nin \Z_p$ implies that the reductions of $p^{-v_p(x_0)}x_0$ and $p^{-v_p(x_0)}x$ are distinct in $\bar \F_p^\times$.

Now assume $x_0$ is general. Since $v_p(x_0 - x') > v_p(x')$, we have $v_p(x_0) = v_p(x')$. By the previous paragraph applied to $x'$, we know that $v_p(x' - x) = v_p(x')$. Thus, $v_p(x_0 - x') > v_p(x' - x)$ as well. But then, the ultrametric inequality implies
\begin{equation*}
v_p(x_0 - x) = v_p(x' - x) = v_p(x') = v_p(x_0),
\end{equation*}
as promised.
\end{proof}
\begin{lemma}\label{lemma:equivalent-conditons-inZp}
If $x_0 \in \cal O_{\C_p}$ then $x_0 \nin \Z_p$ if and only if either:
\begin{enumerate}
\item there exists $x' \in \Z_p$ such that $v_p(x_0-x') \nin \Z \union \set{\infty}$, or
\item there exists $x' \in \Z_p^{\nr} - \Z_p$ such that $v_p(x_0 - x') > v_p(x')$.
\end{enumerate}
\end{lemma}
\begin{proof}
We assume either (a) or (b) holds and we show $\sup_{x \in \Z_p} v_p(x_0-x) < \infty$. If (b) holds, then this is done by Lemma \ref{lemma:Zpnr}. Suppose that (a) holds, and choose such an $x'$ and let $x \in \Z_p$. Then, since $v_p(x_0-x') \nin \Z\union \set{\infty}$ and $v_p(x' - x) \in \Z$ we have 
\begin{equation*}
v_p(x_0 - x) = \min(v_p(x_0-x'),v_p(x'-x)) \leq v_p(x_0-x').
\end{equation*}
We now show the converse.  Specifically, we show that if $x_0 \in \bar \Z_p - \Z_p$ and $v_p(x_0 - x') \in \Z$ for all $x' \in \Z_p$ then (b) holds. By assumption, $v_p(x_0) \in \Z$, and so $x_0' = p^{v_p(x_0)}\omega(\overline{p^{-v_p(x_0)}x_0}) \in \Z_p^{\nr}$ satisfies $v_p(x_0') = v_p(x_0)$ and $v_p(x_0 - x_0') > v_p(x_0)$. If $x_0' \nin \Z_p$ then we are done. Otherwise set $x_1 = x_0 - x_0'$. Then $x_1$ satisfies all the hypotheses imposed on $x_0$ in  this paragraph, and $v_p(x_1) > v_p(x_0)$. Thus we can repeat the construction of $x_1$ from $x_0$, and by induction we can construct an infinite sequence $x_0,x_1,\dotsc$ and $x_0',x_1',\dotsc$ such that $x_{i+1} = x_i - x_i'$ with
\begin{enumerate}[i)]
\item $x_i \in \bar \Z_p - \Z_p$ with  $v_p(x_0) < v_p(x_1) < \dotsb$,  
\item $x_i' \in \Z_p^{\nr}$ with  $v_p(x_i - x_i') > v_p(x_i) = v_p(x_i')$ for all $i\geq 0$,
\end{enumerate}
In particular $x_0 = \sum x_i'$. Since $x_0 \nin \Z_p$ there exists a smallest $i\geq 1$ such that $x_i' \in  \Z_p^{\nr} - \Z_p$. We claim $x_0'-x_i'$ witnesses that (b) is true for $x_0$.

To see that, set $x' = x_0' - x_i' \nin \Z_p$. Since $x_0' \in \Z_p$, by Lemma \ref{lemma:Zpnr}, we have
\begin{equation*}
v_p(x') = \min(v_p(x_i'),v_p(x_0')) = \min(v_p(x_i),v_p(x_0)) = v_p(x_0).
\end{equation*}
Then
\begin{equation*}
v_p(x' - x_0) = v_p(-x_1 - x_i') \geq v_p(x_1) > v_p(x_0) = v_p(x'),
\end{equation*}
as promised.
\end{proof}

\begin{lemma}
\label{lemma:tire_vals}
Suppose that $h \in \Z_p[[w]]$ and the zeros of $h$ are all in $p\Z_p$. Let $w' \in \ideal m_{\C_p}$ such that either
\begin{enumerate}
\item $v_p(w') \nin \Z$, or
\item there exists a $\twid w \in  \Z_p^{\nr} - \Z_p$ such that $v_p(w' - \twid w) > v_p(\twid w)$.
\end{enumerate}
If $r = \floor{v_p(w')}$, then
\begin{multline*}
v_p(h(w')) = v_p(w')\cdot \sizeof \set{w'' \st h(w'') = 0 \text{ and } v_p(w'') \geq r+1} \\
+ \sum_{v=1}^r v\cdot \sizeof\set{w'' \st h(w'') = 0 \text{ and } v_p(w'') = v}.
\end{multline*}
\end{lemma}
\begin{proof}
In either case, if $w'' \in p\Z_p$ then $v_p(w'-w'') = \min(v_p(w'),v_p(w''))$ (see Lemma  \ref{lemma:Zpnr} for case (b)). From this, the statement is immediate.
\end{proof}

The proof of one final lemma is left to the reader.
\begin{lemma}
\label{lemma:easy}
Suppose that $(k_i)$ is an ordered list of integers which form an arithmetic progression of length $M = p^e u$, with $(u,p)=1$, and difference $\delta$ with $(\delta,p)=1$. Then,
\begin{enumerate}
\item $\sizeof\set{k_i \st v_p(k_i) \geq e} = u$, and
\item if $0 \leq v < e$ then $\sizeof \set{k_i \st v_p(k_i) = v} = u\varphi(p^{e-v}) = u(p-1)p^{e-v-1}$.
\end{enumerate}
\end{lemma}

We're now in position to prove Theorem \ref{theorem:global-halo-progressions}.
\begin{proof}[Proof of Theorem \ref{theorem:global-halo-progressions}]
Recall that we assume $w_{\kappa} \nin \Z_p$, we write $\alpha_{\kappa} = \sup_{w \in \Z_p} v_p(w_{\kappa} - w)$, and $r = \floor{\alpha_{\kappa}}$. For notational ease, write $C = C_{p,N,r} = p^r C_{p,N}$. Since $w_{\kappa}$ is not an integer, $\Delta_i(w_{\kappa})$ is well-defined for each $i\geq 1$. Our goal is to compare $v_p(\Delta_i(w_{\kappa}))$ to $v_p(\Delta_{i+C}(w_{\kappa}))$ and then apply Lemma \ref{lemma:newton_slopes}. Write $\Delta_i = \Delta_i^+/\Delta_i^-$ as before.

We first focus on $\Delta_i^+$. By Proposition \ref{prop:lambda_change}(a) we have $\lambda(\Delta_{i+C}^+) = \lambda(\Delta_i^+) + p^{r+1} \cdot {p-1\over 2}$. By Proposition \ref{prop:change_in_zeroes}(a), the zeros of $\Delta_i^+$ (and $\Delta_{i+C}^+$) are of the form $w_{k}$ with $k$ lying in an arithmetic progression of integers whose difference is $p-1$ (save for possibly one zero when $p=3$ and $N=1$). Write $\Delta_{i+C}^+ = a\cdot b$ where $a,b \in \Z_p[[w]]$, $\lambda(b) = p^{r+1}\cdot {p-1\over 2}$, and where the zeros of $a(w)$ are the highest zeros $w = w_{k}$ of $\Delta_{i+C}^+$ for the highest $\lambda(\Delta_i^+)$-many $k$.

Since $w_{\kappa} \nin \Z_p$, $w_{\kappa}$ must satisfy one of the two conditions of Lemma \ref{lemma:equivalent-conditons-inZp}. If (a) is true then choose an integer $k_0$ such that $\alpha_{\kappa} = v_p(w_{\kappa} - w_{k_0}) \nin \Z$, and if (ii) is true then set $w_{k_0} = k_0 = 0$. For each $h \in \set{a,b,\Delta_i^+}$ we then apply Lemma \ref{lemma:tire_vals} to $h(w + w_{k_0})$ and $w' = w_{\kappa} - w_{k_0}$. We deduce (remember $v_p(w_{k}-w_{k_0}) =  1 + v_p(k-k_0)$) that 
\begin{multline}\label{eqn:h-equation-halo}
v_p(h(w_{\kappa})) = \underlabel{\alpha_{\kappa}}{v_p(w_{\kappa}-w_{k_0})} \cdot \sizeof\set{k \st h(w_{k}) = 0 \text{ and } v_p(k-k_0) \geq r}\\
+ \sum_{v=0}^{r-1} (v+1) \cdot \sizeof \set{k \st h(w_{k}) = 0 \text{ and } v_p(k-k_0) = v}.
\end{multline}
Now we claim that $v_p(a(w_{\kappa})) = v_p(\Delta_i^+(w_{\kappa}))$. If $\lambda(\Delta_i^+) = 0$ then there is nothing to show. Otherwise, if $\lambda(\Delta_i^+) > 0$ then Lemma  \ref{lemma:lambdazp_change}(a) implies that
\begin{equation*}
\HZ(a) = \HZ(\Delta_{i+C}^+) \congruent \HZ(\Delta_i^+) \bmod p^{r+1}.
\end{equation*}
Since the $k$ for which $w_{k}$ is a zero of either $\Delta_i^+$ or $a$ is an arithmetic progression, and the last terms are congruent modulo $p^{r+1}$ (as we just checked), we see that the right-hand side of \eqref{eqn:h-equation-halo} is the same for $h = a$ and $h=\Delta_i^+$. (The reader can check that the single missing zero when $p=3$ and $N=1$ does not affect this argument.)

On the other hand, the zeros of $b$ are $w = w_{k}$ with $k$ lying in an arithmetic progression of length $M = p^{r+1}{p-1\over 2}$ and difference $p-1$. Thus it follows from the previous paragraph, Lemma \ref{lemma:easy} and \eqref{eqn:h-equation-halo} that
\begin{align*}
v_p\left(\Delta_{i+C}^+(w_{\kappa})\over \Delta_i^+(w_{\kappa})\right) &= v_p(b(w_{\kappa})) \\
&= \alpha_{\kappa}\cdot \left({p-1\over 2}
+{p-1\over 2}(p-1)\right)
 + \sum_{v=0}^{r-1} (v+1) \cdot {p-1\over 2}(p-1) p^{r-v}\\
 &= p\cdot {p-1\over 2}\left(\alpha_{\kappa} + \sum_{v=1}^r v\cdot (p-1)p^{r-v}\right).
\end{align*}
An analogous computation shows that
\begin{equation*}
v_p\left({\Delta_{i+C}^-(w_{\kappa})\over \Delta_i^-(w_{\kappa})}\right) = {p-1\over 2}\left(\alpha_{\kappa} + \sum_{v=1}^r v\cdot (p-1)p^{r-v}\right)
\end{equation*}
Combining the previous two equations, we deduce
\begin{equation*}
v_p\left({\Delta_{i+C}(w_{\kappa})\over \Delta_i(w_{\kappa})}\right) = {(p-1)^2\over 2}\left(\alpha_{\kappa} + \sum_{v=1}^r (p-1)p^{r-v}\cdot v\right)
\end{equation*}
This shows that the $\Delta$-slopes form a union of $C$ arithmetic progressions whose common difference is our claimed one. Our theorem then follows from Lemma \ref{lemma:newton_slopes}.
\end{proof}

\begin{remark}\label{rmk:p=2-aps-remark}
One can ask for a version of Theorem \ref{theorem:global-halo-progressions} valid if $p=2$.  If $N=1$ then it is not difficult to establish an analog of Theorem \ref{theorem:global-halo-progressions}.  Namely, if $\alpha_{\kappa} < 3$ then the slopes of $\NP(G_\kappa)$ is $\set{i \cdot v_2(w_{\kappa}) \st i = 1,2,\dotsc}$ (Theorem \ref{theorem:agree-BK-p=2}) and thus a single arithmetic progression with common difference $v_2(w_{\kappa})$. If $\alpha_{\kappa} \geq 3$ and $r = \floor{\alpha_\kappa}$ then one may also show:\ except for finitely many exceptional slopes, the slopes of $\NP(G_{\kappa})$ are a finite union of $2^{r - 2}$-many arithmetic progressions whose common difference is
\begin{equation*}
\alpha_\kappa  + \sum_{v=3}^{r} v \cdot 2^{r-v}.
\end{equation*}
The proof is analogous to the above, using Proposition \ref{proposition:explicit-p=2-info} for explicit analogs of Lemma \ref{lemma:lambdazp_change}, Proposition \ref{prop:lambda_change}, etc.

One could also ask about $N > 1$. But, since the ghost series requires modification in that case (Section \ref{sec:modification} below) we did not pursue this.
\end{remark}

\section{A $2$-adic modification for the ghost series}\label{subsec:modification}\label{sec:modification}
In this section we construct a modification of the ghost series which we conjecture determines slopes when $p=2$ is $\Gamma_0(N)$-regular (Conjecture \ref{conj:p=2} below). The theme of this section is that non-integral slopes are forced to be repeated and this should be taken into account in the ghost series.

We emphasize that $N$ is an odd positive integer in this section. Recall \cite[Definition 1.3]{Buzzard-SlopeQuestions}:
\begin{definition}\label{defn:2regular} 
The prime $p=2$ is called $\Gamma_0(N)$-regular if 
\begin{enumerate}
\item The eigenvalues $T_2$ acting on $S_2(\Gamma_0(N))$ are all $2$-adic units and 
\item The slopes of $T_2$ acting on $S_4(\Gamma_0(N))$ are all either zero or one.
\end{enumerate}
\end{definition}
Our definition is equivalent to \cite[Definition 1.3]{Buzzard-SlopeQuestions} by Hida theory. Also by Hida theory,
\begin{multline}\label{eqn:inequality-ord-dime}
\dim S_2(\Gamma_0(2N))^{\ord} \leq \dim S_2(\Gamma_0(N)) + \dim S_{2}(\Gamma_0(2N))^{2-\new}\\
 = \dim S_2(\Gamma_0(2N)) - \dim S_2(\Gamma_0(N))
\end{multline}
with equality if $p=2$ is $\Gamma_0(N)$-regular.

We now produce non-integral slopes for $U_2$ acting on certain spaces with quadratic character regardless of a regularity hypothesis.\footnote{Note:\ not in spaces $S_k(\Gamma_0(2N))$ which would contradict Buzzard's conjecture.} Write $\eta_8^{\pm}$ for the Dirichlet characters of conductor $8$ with sign $\pm$. The character $\eta_8^{\pm}$ is quadratic, so the slopes of $U_2$ acting on $S_k(\Gamma_0(N)\cap \Gamma_1(8),\eta_8^{\pm})$ are symmetric around ${k-1\over 2}$. Hida theory implies that 
\begin{equation}\label{eqn:p=2_hidacons}
\dim S_2(\Gamma_0(N)\intersect \Gamma_1(8),\eta_8^+)^{\set{0,1}} = 2\dim S_2(\Gamma_0(2N))^{\ord}.
\end{equation}
(Here and below, if $S$ is a set of cuspforms and $X$ is a set of real numbers then we write $S^X$ for the subspace spanned by eigenforms whose slope lies in $X$.)

\begin{proposition}\label{prop:fractional-slopes-p=2}
If $N > 1$ is odd then $\dim S_2(\Gamma_0(N)\intersect \Gamma_1(8),\eta_8^+)^{(0,1)} > 0$.
\end{proposition}
\begin{proof}
By \eqref{eqn:inequality-ord-dime} and \eqref{eqn:p=2_hidacons}, we see
\begin{multline*}
\dim S_2(\Gamma_0(N)\intersect \Gamma_1(8),\eta_8^+)^{(0,1)} \geq \\
\dim S_2(\Gamma_0(N)\intersect \Gamma_1(8),\eta_8^+) - 2\bigl(\dim S_2(\Gamma_0(2N)) - \dim S_2(\Gamma_0(N))\bigr).
\end{multline*}
The final expression is positive if $N > 1$ (see Lemma \ref{app-lemma:second-dim-wt-2-cond-8}).
\end{proof}

Since the characters $\eta_8^{\pm}$ have values only $\pm 1$, any non-integral slope appearing in a space $S_k(\Gamma_0(N)\cap \Gamma_1(8),\eta_8^{\pm})$ must be repeated. In particular, Proposition \ref{prop:fractional-slopes-p=2} implies that for $N > 1$, there exists non-integral repeated slopes in $S_2(\Gamma_0(N)\cap \Gamma_1(8),\eta_8^+)$. The ghost series defined thus far {\em does not} see these slopes:\
\begin{example}\label{example:p=2N=3}
$p=2$ is $\Gamma_0(3)$-regular since there are no forms of weight two or four. The space $S_2(\Gamma_0(3)\intersect \Gamma_1(8), \eta_8^+)$ is two-dimensional with slope $1/2$ repeated twice. On the other hand, the ghost series predicts slopes zero and one (see Example \ref{example:issue-p=2}).
\end{example}

Our goal now is to salvage the ghost conjecture for $p=2$ by including the fractional (repeated) slopes appearing in the spaces $S_k(\Gamma_0(N)\cap \Gamma_1(8),\eta_8^{\pm})$ as $k$ varies and $\pm = (-1)^k$ (we use this implicit notation throughout). Specifically, for each integer $k\geq 2$, we are going to define a second multiplicity pattern $m^{\circ}(k) = (m_i^{\circ}(k))$ which will describe the multiplicity of the weight $z^k\eta_8^{\pm}$ as a zero of a modified ghost series. Our model will be 
\begin{multline}\label{eqn:mi-circ-hope}
m_i^{\circ}(k) > 0
\iff \text{the $i$-th and $(i+1)$-st slope in $S_k(\Gamma_0(N)\cap \Gamma_1(8),\eta_8^{\pm})$}\\\text{are the same and strictly between $k-2$ and $k-1$.}
 \end{multline}
In fact, for each $i$ there will be at most one $k$ such that $m_i^{\circ}(k)$ is positive. Granting the definition of $m_i^{\circ}(k)$, we then define
\begin{equation*}
g_i^{\circ}(w) = g_i(w) \cdot \prod_{k=2}^\infty (w - w_{z^k\eta_8^{\pm}})^{m_i^{\circ}(k)}
\end{equation*}
and the {\em modified ghost series} $G^{\circ}(w,t) = 1 + \sum g_i^{\circ}(w)t^i$. It is still an entire series over $\Z_2[[w]]$ (since we've only added more zeros). If $N = 1$ then $G = G^{\circ}$.
\begin{conjecture}\label{conj:p=2}
If $p=2$ is $\Gamma_0(N)$-regular then $\NP(G_\kappa^{\circ}) = \NP(P_\kappa)$ for each $\kappa \in \cal W$.
\end{conjecture}
We briefly give the evidence we have for Conjecture \ref{conj:p=2}. Recall we write $\BS(k)$ for the output of Buzzard's algorithm in weight $k$. The levels $N$ in Theorem \ref{theorem:p=2_evidence} below are all the levels $N\leq 167$ such that $p=2$ is $\Gamma_0(N)$-regular. The next $N$ is $191$.\footnote{If $p=2$ is $\Gamma_0(N)$-regular then must $N$ be either 1,3 or be a prime congruent to $7 \bmod 8$? Anna Medvedovsky tells us that $p=2$ is not $\Gamma_0(\ell)$-regular when $\ell > 3$ is a prime $3\bmod 8$.}
\begin{theorem}\label{theorem:p=2_evidence}
If $N = 3,7,23,31$ then $\NP((G_k^{\circ})^{\leq d_k}) = \BS(k)$ for all even $k\leq 5000$, or if $N = 47, 71, 103, 127, 151, 167$ then $\NP((G_k^{\circ})^{\leq d_k}) = \BS(k)$ for all even $k\leq 2050$.
\end{theorem}
\begin{remark}
One could ask about the asymptotic results in Section \ref{subsec:distribution}. As we will see, for each $i$, the total multiplicity $\sum_k m_i^{\circ}(k)$ of zeros of $g_i^{\circ}$ which were not a zero of $g_i$ is bounded, and the extra zeros added are at weights $\kappa$ which satisfy $v_2(w_{k} - w_{\kappa}) = 1$ for all $k \in \Z$. Thus the estimates in Section \ref{subsec:distribution} will only be effected by $O(1)$ terms and so Corollary \ref{cor:highest} and Corollary \ref{corollary:gouvea-distribution} should still hold with $G(w,t)$ replaced by $G^{\circ}(w,t)$.
\end{remark}

The rest of this section is devoted to describing the multiplicity $m_i^{\circ}(k)$ of $w_{z^k\eta_8^{\pm}}$ as a zero of $g_i^{\circ}$. The idea is to {\em force} the issue for $m_i^{\circ}(2)$ by insisting that \eqref{eqn:mi-circ-hope} holds, and that the precise value of $m_i^{\circ}(2)$ follows the up-down pattern within the indices which realize each fractional slope. We then extend the pattern to $k > 2$ ``using the spectral halo'' (Section \ref{subsec:ghost-spectral-intro}).

More precisely, if $k\geq 2$ write
\begin{equation*}
d_k^{\circ} := \dim S_k(\Gamma_0(N)\intersect \Gamma_1(8),\eta_8^{\pm}).
\end{equation*}
Write $\nu_1^{\circ}(2) \leq \nu_2^{\circ}(2) \leq \dotsb \leq \nu_{d_2^{\circ}}^{\circ}(2)$ for the list of slopes of $U_2$ acting on $S_2(\Gamma_0(N)\intersect \Gamma_1(8),\eta_8^+)$.  Write $\nu_{i_1}^{\circ}(2) < \nu_{i_2}^{\circ}(2) < \dotsb < \nu_{i_t}^{\circ}(2)$ for the {\em distinct} slopes appearing in this list where $i_j$ is the least $i$ such that $\nu_{i_j}^{\circ}(2) = \nu_i^{\circ}(2)$ (so $i_1 = 1$). Also set $i_0 = 0$, $i_{t+1} = d_2^{\circ}$, and $\mu_{j}$ for the multiplicity of $\nu_{i_j}^{\circ}(2)$ among the $\nu_i^{\circ}(2)$. Then set
\begin{equation*}
m_i^{\circ}(2) = \begin{cases}
s_{i}(\mu_{j}-1,i_{j}-1) & \text{if $i_{j} \leq i < i_{j+1}$ for some $1 \leq j \leq t$ \text{and $\nu_{i_j}^{\circ}(2)\neq 0,1$}}\\
 0 & \text{otherwise.}\\
\end{cases}
\end{equation*}
where $s_i(\ast,\ast)$ is the up-down pattern from Section \ref{subsec:ghost-conjecture}. We give three examples ($p=2$ is $\Gamma_0(N)$-regular for each $N$ below):

\begin{example}
Let $N = 3$. Then the slopes are computed in Example \ref{example:p=2N=3}, and we have $\nu_{1}^{\circ}(2) = \nu_{2}^{\circ}(2) = {1\over 2}$. Thus $t = 1$, $i_t = i_1 = 1$ and $(m_i^{\circ}(2) \st i \geq 1) = (1,0,0,0,\dotsc)$.
\end{example}
\begin{example}
Let $N = 7$. The slopes of $U_2$ acting on $S_2(\Gamma_0(7)\intersect \Gamma_1(8),\eta_8^+)$ are $[0,{1\over 2},{1\over 2},{1\over 2},{1\over 2},1]$. We have 
\begin{equation*}
\nu_1^{\circ}(2) = 0 < \nu_2^{\circ}(2) = \dotsb \nu_5^{\circ}(2) = {1\over 2} < \nu_6^{\circ}(2) = 1.
\end{equation*}
Thus $t = 3$, $(i_1,i_2,i_3) = (1,2,6)$ and $(m_i^{\circ}(2) \st i \geq 1) = (0,1,2,1,0,0,\dotsc)$.
\end{example}

\begin{example}
Let $N = 23$. Then the slopes are $[0_3, ({1\over 3})_6, ({1\over 2})_4, ({2 \over 3})_6, 1_3]$ (the subscripts refer to the multiplicity). The sequence $m_i^{\circ}(2)$ is given by
\begin{equation*}
(m_i^{\circ}(2) \st i \geq 1) = (0,0,0,1,2,3,2,1,0,1,2,1,0,1,2,3,2,1,0,0,0,0,0,0,0,\dotsc).
\end{equation*}
\end{example}

Now, if $k > 2$ then we will set 
\begin{equation}\label{eqn:p=2-mults}
m_i^{\circ}(k) = 
\begin{cases}
m_{d_k^{\circ}-i}^{\circ}(2) & \text{if $1 \leq i < d_k^{\circ}$}\\
0 & \text{otherwise.}
\end{cases}
\end{equation}
This completes the definition of the $m_i^{\circ}(k)$ and completes the statement of Conjecture \ref{conj:p=2}. 

The rest of this section is devoted to expanding on the definition of $m_i^{\circ}(k)$ when $k>2$. First, the authors believe that a version of the spectral halo will imply that $k \mapsto \NP(P_{z^k\eta_8^{\pm}})$ is independent of $k$. In particular, if our {\em modus operandi} is to predict the fractional slopes appearing in $S_k(\Gamma_0(N)\cap \Gamma_1(8),\eta_8^{\pm})$ then we should restrict to slopes between $k-2$ and $k-1$ (the lower slopes being correctly predicted ``by induction'' on $k$; compare with Remark \ref{remark:p2zerosremark}).

Now, when is the $i$-th and $(i+1)$-st slope going to be more than $k-2$ and not more than $k-1$? First, ``by the spectral halo'' we should certainly have $d_{k-1}^{\circ} < i$. But, there are also the $c_0(N)$-many $\theta^{k-2}$-critical Eisenstein series which are overconvergent $p$-adic cuspforms of weight $z^{k-1}\eta_8^{\pm}$ and slope $k-2$ ($c_0(N)$ being the number of cusps of $X_0(N)$). Thus if we want the $i$-th and $(i+1)$-st slope to be larger than $k-2$, we should expect $d_{k-1}^{\circ} + c_0(N) < i$. We now note the following lemma.
\begin{lemma}
If $k > 2$ and $d_{k-1}^{\circ} + c_0(N) < i$ then $d_k^{\circ}-i < d_2^{\circ}$.
\end{lemma}
\begin{proof}
By Lemma \ref{app-lemma:dim-wt-2-cond-8} we have $d_2^{\circ} = \mu_0(N) - c_0(N)$ and if $k > 2$ then $d_k^{\circ} = \mu_0(N) + d_{k-1}^{\circ}$. The lemma clearly follows then.
\end{proof}
Now write $\nu_1^{\circ}(k) \leq \nu_2^{\circ}(k) \leq  \dotsb$ for the slopes of $U_2$ acting on $S_k(\Gamma_0(N)\cap \Gamma_1(8),\eta_8^{\pm})$. Since $\eta_8^{\pm}$ is quadratic, the Atkin--Lehner involution implies that
\begin{equation*}
\nu_i^{\circ}(k) = \nu_{i+1}^{\circ}(k) \text{ is in $(k-2,k-1)$} \iff \nu_{d_k^{\circ}-i}^{\circ}(k) = \nu_{d_k^{\circ}-i+1}^{\circ}(k)  \text{ is in $(0,1)$.}
\end{equation*}
``By the spectral halo'', we have an equivalence 
\begin{equation*}
\nu_{d_k^{\circ}-i}^{\circ}(k) = \nu_{d_k^{\circ}-i+1}^{\circ}(k)  \text{ is in $(0,1)$} \iff \nu_{d_k^{\circ}-i}^{\circ}(2) = \nu_{d_k^{\circ}-i+1}^{\circ}(2)  \text{ is in $(0,1)$.}
\end{equation*}
We just justified that our natural constraint on $i$ should force $d_k^{\circ} - i < d_2^{\circ}$, so that the right-hand side of the previous equivalence exactly describes when $m_{d_k^{\circ}-i}(2) > 0$, and strongly suggests the definition \eqref{eqn:p=2-mults} is natural.

\begin{remark}\label{remark:p2zerosremark}
It is not hard to see that if $1 \leq i < \infty$ then there exists at most one $k$ for which $m_i^{\circ}(k) > 0$, so we can then write $m_i^{\circ\circ}$ for this non-zero value, if it exists. Based on our heuristic of using the spectral halo, one could also form an alternate modification
\begin{equation*}
g_i^{\circ\circ}(w) = g_i(w)(w-w_{z^2\eta_8^{\pm}})^{m_i^{\circ\circ}}
\end{equation*}
by adding a zero at the single weight $\kappa = z^2\eta_8^{\pm}$ infinitely often. Then one could form an alternate modified ghost series $G^{\circ\circ}(w,t) = 1 + \sum g_i^{\circ\circ}(w)t^i$. Numerical checks suggests that $\NP(G^{\circ}_\kappa) = \NP(G^{\circ\circ}_\kappa)$ for all $\kappa$, but we will not pursue proving that here. It is certainly true if $v_2(w_\kappa) > 1$.
\end{remark}

\begin{appendix}

\section{Dimension formulas}
\label{app:dimension-formulae}

The goal of this appendix is to gather together various estimates and formulas for the dimensions of spaces of cuspforms. The results for spaces with trivial character are deduced from the standard formulas in \cite[Section 6.1]{Stein-ModularForms}.\footnote{Freely available at 
\href{http://wstein.org/books/modform/modform/dimension_formulas.html}{{\tt http://wstein.org/books/modform/modform/dimension\textunderscore formulas.html}.}}
We use the notation(s):\ $\mu_0(N)$ for the index of $\Gamma_0(N)$ in $\SL_2(\Z)$, $c_0(N)$ for the number of cusps of $X_0(N)$, $\mu_{0,2}(N)$ for the number of elliptic points of order two on $X_0(N)$, $\mu_{0,3}(N)$ for the number of elliptic points of order three on $X_0(N)$ and $g_0(N)$ for the genus of $X_0(N)$. Many proofs are asymptotically clear and we often leave the details of explicit constants to the reader.

We fix $N$ and $p$ throughout the appendix and we will also assume that $p \ndvd N$ as a rule. As in the main text we write $d_k = \dim S_k(\Gamma_0(N))$,  $d_{k,p} = \dim S_k(\Gamma_0(Np))$ and $d_k^{\new} = \dim S_k(\Gamma_0(Np))^{p-\new}$. For example, if $k > 2$ is even then
\begin{equation}\label{eqn:dim-formula-old}
d_k = (k-1)(g_0(N) - 1) + \left({k\over 2} - 1\right)c_0(N) + \bfloor{k\over 4}\mu_{0,2}(N) + \bfloor{k\over 3}\mu_{0,3}(N)
\end{equation}
where $g_0(N)$ may be written as 
\begin{equation*}
d_2 = g_0(N) = 1 + {\mu_0(N)\over 12} - {\mu_{0,2}(N) \over 4} - {\mu_{0,3}(N) \over 3} - {c_0(N)\over 2}.
\end{equation*}
To compute $d_{k,p}$, one replaces $N$ by $Np$ everywhere in \eqref{eqn:dim-formula-old}. Then it is easy to check that for $p \ndvd N$ and $k>2$ then 
\begin{multline}\label{eqn:dim-formula-new}
d_k^{\new} = {(k-1)(p-1)\over 12}\mu_0(N) + \left(\bfloor{{k\over 4}} - {k-1\over 4}\right)\left(-1 + \left({-4 \over p}\right)\right)\mu_{0,2}(N)\\ + \left(\bfloor{{k\over 3}} - {k-1\over 3}\right)\left(-1 + \left({-3 \over p}\right)\right)\mu_{0,3}(N),
\end{multline}
where $\left({a \over b}\right)$ is the Kronecker symbol.

\begin{lemma}\label{mu:inequalities-need-for-dimensions}
If $N > 1$ then 
$\displaystyle 
{1\over 6}\mu_0(N) - {1\over 2}\mu_{0,2}(N) - {2\over 3}\mu_{0,3}(N) \geq 0.
$
\end{lemma}
\begin{proof}
Let $\omega(N)$ denote the number of distinct prime divisors of $N$. Then $\omega(N) \leq \log_3(N)$ if $N\geq 6$ and $\mu_{0,i}(N) \leq 2^{\omega(N)}$ for $i=2,3$. Moreover, if $N \geq 200$ then $N \geq 7\cdot  2^{\log_3(N)}$. Thus for $N\geq 200$ we conclude
\begin{equation*}
\mu_0(N) \geq N \geq 7 \cdot 2^{\log_3(N)} \geq 7\cdot 2^{\omega(N)} \geq 6\cdot\left({1\over 2}\mu_{0,2}(N) + {2\over 3}\mu_{0,3}(N)\right).
\end{equation*}
We leave checking $2 \leq N \leq 200$ for the reader (or a computer).
\end{proof}

\begin{lemma}\label{lemma:dim-formulas-increasing}
\leavevmode
\begin{enumerate}
\item If $N > 1$ then $k \mapsto d_k$ is a weakly increasing function of even weights $k\geq 2$.
\item If $N = 1$ and $p>3$ then $n \mapsto d_{k+n(p-1)}(\SL_2\Z)$ is increasing.
\item If $N > 1$ or $p > 3$ then $d_k^{\new} \leq d_{k + \varphi(2p)}^{\new}$
\end{enumerate}
\end{lemma}
\begin{proof}
For (a) it is clear from \eqref{eqn:dim-formula-old} that if we restrict to $k\geq 4$ and either $g_0(N) \geq 1$ or $c_0(N) \geq 2$. That leaves $N=1$, which we've excluded, and checking $d_2 \leq d_4$ (which is easy).

Part (b) is also easy. If $N=1$, $j\geq 0$ and $d_k > d_{k+j}$ then $k \congruent 0 \bmod 12$ and $j = 2$. In particular, if $p$ is odd and $d_k > d_{k+(p-1)}$ then $p=3$.

Let's prove (c). First, using \eqref{eqn:dim-formula-new} to compute $d_{k+\varphi(2p)}^{\new} - d_k^{\new}$, one uniformly sees that if $k\geq 4$ then
\begin{equation}\label{eqn:difference-equation}
d_{k+\varphi(2p)}^{\new} - d_k^{\new} \geq {\varphi(2p)(p-1) \over 2}\mu_0(N) - \mu_{0,2}(N) - {4\over 3}\mu_{0,3}(N).
\end{equation}
If $N > 1$ then Lemma \ref{mu:inequalities-need-for-dimensions} implies the right-hand side is $\geq 0$ as long as $p \geq 7$. We leave the remaining cases of $p=2,3,5$ and $N>1$, $N = 1$ and $p>3$, and $k=2$ to the reader (one just needs to make the lower bound \eqref{eqn:difference-equation} more explicit.)
\end{proof}

\begin{lemma}\label{lemma:weight2-inequality}
If $p$ is odd and $n\geq 1$ then $d_{2 + n(p-1)} \geq d_2 + d_2^{\new}$ with equality if $n = 1$.
\end{lemma}
\begin{proof}
Let's first show equality for $n=1$. One computes
\begin{multline}\label{eqn:difference-blah}
d_{2+(p-1)} - (d_2 + d_2^{\new})\\= \left(\bfloor{p+1\over 4} - {p \over 4} + {1\over 4}\left({-4 \over p}\right)\right)\mu_{0,2}(N) + \left(\bfloor{p+1\over 3} - {p \over 3} + {1\over 3}\left({-3 \over p}\right)\right)\mu_{0,3}(N).
\end{multline}
The right-hand side of \eqref{eqn:difference-blah} clearly vanishes for all odd $p$. Next, Lemma \ref{lemma:dim-formulas-increasing} allows us to finish except if $N=1$ and $p=3$, where the result is trivial anyways because $d_2 + d_2^{\new} = 0$.
\end{proof}

\begin{lemma}\label{lemma:appendix-2adic-wt-2}
Let $p=2$. If $n\geq 1$ then $d_{2(n+1)} \geq d_2 + d_2^{\new}$ with equality when $n=1$ only if $N=1,3,7$.
\end{lemma}
\begin{proof}
One checks explicitly that 
\begin{equation*}
d_4 - (d_2 + d_2^{\new}) = {1 \over 12}\mu_0(N) + {1\over 4}\mu_{0,2}(N) - {1\over 3}\mu_{0,3}(N),
\end{equation*}
and this is equal to zero if and only if $N = 1,3,7$ by an argument similar to Lemma \ref{mu:inequalities-need-for-dimensions}. If $N > 1$ then we are finished by Lemma \ref{lemma:dim-formulas-increasing}. If $N = 1$ then $d_2 + d_2^{\new} = 0$ so the result is trivial in that case.
\end{proof}
\begin{lemma}\label{app-lemma:shift-wt-p(p-1)}
If $p\geq 5$ is odd, $k\geq 4$ is even and $j \geq 0$ then
$$
\displaystyle d_{k+j(p-1),p} - d_{k,p}
 = {j(p-1)(p+1)\over 12}\mu_0(N).
$$ If $p=3$ then the same holds for $j \congruent 0 \bmod 3$.
\end{lemma}
\begin{proof}
Let $p\geq 3$. Then,
\begin{equation*}
\left(\bfloor{{k + j(p-1)\over 4}} - \bfloor{{k\over 4}}\right)\mu_{0,2}(Np) = {j(p-1)\over 4}\mu_{0,2}(Np).
\end{equation*}
If $p\congruent 1 \bmod 4$ this is clear, and if $p \congruent 3 \bmod 4$ then both sides vanish because $\mu_{0,2}(Np) = 0$. Similarly, if either $p \geq 5$ or if $j \congruent 0 \bmod 3$ then
\begin{equation*}
\left(\bfloor{{k + j(p-1)\over 3}} - \bfloor{{k\over 3}}\right)\mu_{0,3}(Np) = {j(p-1)\over 3}\mu_{0,3}(Np).
\end{equation*}
Thus,
\begin{multline*}
d_{k+j(p-1),p} - d_k\\
 = j(p-1)(g_0(Np)-1) + {{j(p-1)\over 2}c_0(Np)} + {j(p-1)\over 4}\mu_{0,2}(Np) + {j(p-1)\over 3}\mu_{0,3}(Np)\\
 = {j(p-1)\over 12}\underlabel{\mu_0(Np)}{\left(12(g_0(Np)-1) + 6c_0(Np) + 3\mu_{0,2}(Np) + 4\mu_{0,3}(Np)\right)}.
\end{multline*}
Since $\mu_0(Np) = (1+p)\mu_0(N)$ we are done.
\end{proof}
\begin{lemma}\label{lemma:robs-tedious-calculation}
If $p\geq 5$, $k\geq 4$ is even and $j\geq 0$ then
\begin{equation*}
d_{k+j(p-1)} - d_k + \bfloor{{d_{k+j(p-1)}^{\new}\over 2}} - \bfloor{{d_k^{\new}\over 2}} = {j(p-1)(p+1)\over 24}\mu_0(N).
\end{equation*}
If $p = 3$ then the same is true of $j \congruent 0 \bmod 3$.
\end{lemma}
\begin{proof}
Since $d_k^{\new} \congruent d_{k,p} \bmod 2$,
Lemma \ref{app-lemma:shift-wt-p(p-1)} implies that $d_{k+j(p-1)}^{\new} \congruent d_k^{\new} \bmod 2$. Thus,
\begin{multline*}
d_{k+j(p-1)} - d_k + \bfloor{{d_{k+j(p-1)}^{\new}\over 2}} - \bfloor{{d_k^{\new}\over 2}} \\ = d_{k+j(p-1)} - d_k + {1\over 2}(d_{k+j(p-1)}^{\new} - d_k^{\new})
 = {1\over 2}(d_{k+j(p-1),p} - d_{k,p})
\end{multline*}
Thus Lemma \ref{app-lemma:shift-wt-p(p-1)} finishes the proof.
\end{proof}

We finish with formulas for spaces with character. For this we use Cohen--Oesterl\'e  \cite{CohenOesterle-Dimensions}.

\begin{lemma}\label{app-lemma:dim-wt-2-cond-8}
If $N \geq 1$ is odd, $k\geq 2$ is even and $\eta_8^{\pm}$ is the primitive character modulo $8$ such that $\eta_8^{\pm}(-1) = (-1)^k$ then
\begin{equation*}
\dim S_k(\Gamma_0(N)\intersect \Gamma_1(8),\varepsilon) = (k-1)\mu_0(N) - c_0(N).
\end{equation*}
\end{lemma}
\begin{proof}
This is immediate from \cite[Th\'eor\`eme 1]{CohenOesterle-Dimensions}. One should take $8N$ for $N$ in the reference, $\chi=\eta_8^{\pm}$ and note remark $1^{\circ}$ in {\em loc. cit.}
\end{proof}

\begin{lemma}\label{app-lemma:second-dim-wt-2-cond-8}
If $N\geq 1$ is odd and $\eta_8^+$ is the even primitive character modulo $8$ then
\begin{equation*}
\dim S_2(\Gamma_0(N)\intersect \Gamma_1(8),\eta_8^+) - 2\bigl(\dim S_2(\Gamma_0(2N)) - \dim S_2(\Gamma_0(N))\bigr) = {2 \over 3}\left(\mu_0(N) - \mu_{0,3}(N)\right).
\end{equation*}
In particular, it is positive if and only if $N > 1$.
\end{lemma}
\begin{proof}
One computes explicitly that 
\begin{align*}
\dim S_2(\Gamma_0(2N)) - \dim S_2(\Gamma_0(N)) = {\mu_0(N)\over 6} + {\mu_{0,3}(N)\over 3} - {c_0(N)\over 2}.
\end{align*}
The equality then follows from Lemma \ref{app-lemma:dim-wt-2-cond-8} applied with $k=2$. Regarding the positivity, if $N = 1$ then $\mu_0(N) = \mu_{0,3}(N) = 1$. To check it is positive if $N > 1$, it reduces to a finite computation (which we leave to the reader).
\end{proof}

\section{The story of an explicit calculation when $p=2$ and $N = 1$}
\label{app:explicit-2adic}

The goal of this appendix is to describe the $2$-adic calculation we made which motivated the  multiplicity pattern and the use of ``ghost'' in ``the ghost conjecture''. We begin by quoting an unpublished note of Buzzard.\footnote{Page 2 of the note ``Explicit formulae..." at \href{http://wwwf.imperial.ac.uk/~buzzard/maths/research/notes/}{{\tt http://wwwf.imperial.ac.uk/\textasciitilde buzzard/maths/research/notes/}}} In it, he writes:
\begin{quote}
``...the trace [of $U_2$ acting on overconvergent $2$-adic cuspforms] vanishes at weight $w = 2^3+2^5+2^6+2^7+2^8+2^{13}+2^{16}+2^{18}+2^{19}+\dotsb$, and this corresponds to $k = 2+2^2 +2^3 +2^{11} +2^{15} +2^{16} +2^{18} + \dotsb$, which, unsurprisingly, is close to 14."
\end{quote}
Indeed, $S_{14}(\Gamma_0(2))$ has two distinct eigenforms, both new at 2 and whose $U_2$-eigenvalues are $6$ and $-6$. Thus $\tr(\restrict{U_2}{S_{14}(\Gamma_0(2))}) = 0$ and  $\tr(\restrict{U_2}{S_{14}^{\dagger}(\Gamma_0(2))}) \congruent 0 \bmod 2^{13}$. One can check  that the zero $w = 2^3 + 2^5 + \dotsb$ satisfies $v_2(w_{14}-w) = 13$. In this way, $w_{14}$ is a ``ghost zero'' of the trace:\ it is an integer weight and the true zero of the trace is only a slight $2$-adic deformation.

In order to investigate whether the above phenomenon generalizes, we implemented Koike's formula \cite{Koike-padicProperties} on a computer and computed the first twenty coefficients of $P_\kappa(t) = 1 + \sum a_i(w_\kappa)t^i$ (see  \cite{Robwebsite}). For each $i \leq 20$ we noticed that if $a_i(w_0) = 0$ then $v_2(w_0) \in \Z$ (see \cite[Appendix B]{BergdallPollack-FredholmSlopes}). Thus, it seems possible that the roots of the $a_i$ are relatively near actual integer weights $w_{k}$. And, we conjectured that for some meaning of ``relatively near'', the $w_{k}$ could be taken so that the $i$-th and $(i+1)$-st slope in weight $k$ is ${k-2\over 2}$, i.e.\  $k=6i+8,6i+10,\dotsc,12i-2,12i+2$ (Proposition \ref{proposition:explicit-p=2-info}).

Let's see how this works out. We just pointed out that the unique zero of $a_1$ lies on $v_2(w-w_{14}) = 13$. For $a_2$, the predicted ghost zeros are $w_{20}, w_{22}$ and $w_{26}$. In Table \ref{table:b2-zero-location} below we give the relative position of the zeros of $a_2$ to these three weights. We see what we want:\ the true zeros of $a_2$ are slight $2$-adic deformations of $w_{k}$ with $k=20,22,26$. Similarly, one can work out that the six weights $w_{26},w_{30},\dotsc,w_{34},w_{38}$ are ghost zeros for the third coefficient (which has six zeros).

 \begin{table}[htpp]
\caption{Relative position of zeros of $a_2(w)$ to the weights $w_{k}$ for $k=20,22,26$. (Bold indicates the witnesses to $k$ as a ``ghost zero''.)}
\begin{center}
\begin{tabular}{|c|c|c|c|}
\hline
$k$ & $20$ & $22$ & $26$\\
\hline
$v_2(w_0-w_{k}) \st a_2(w_0) = 0$ & $\mathbf{12}, 3,3$& $\mathbf{13}, 4, 3$ &  $\mathbf{9}, 4, 3$\\
\hline
\end{tabular}
\end{center}
\label{table:b2-zero-location}
\end{table}

A departure must occur for the fourth coefficient:\ $a_4$ has ten zeros and there are only nine predicted ghost zeros. The relative position of the ten zeros to the nine predictions are given in Table \ref{table:b4-zero-location}. What we see is that for each $k=32,34,\dotsc,46,50$ there is a small $2$-adic disc around $w_{k}$ containing at least one root of $a_4$, and that there are actually two roots in a small disc around $w_{38}$. In this sense, $38$ is a ghost zero for $a_4$ with multiplicity two and the rest of the $w_{k}$ have multiplicity one.

 \begin{table}[htpp]
\caption{Relative position of zeros of $a_4(w)$ to the weights $w_{k}$ for $k=32,34,\dotsc,46,50$. (Bold indicates the witnesses to $k$ as a ``ghost zero''.) }
\begin{center}
\begin{tabular}{|c|l|}
\hline
$k$ &  $v_2(w_0 - w_{k})$ where $a_4(w_0) = 0$\\
\hline
 $32 $&$\mathbf{9}, 5, 4, 4, 3,\dotsc$\\
  $34 $&$ \mathbf{9}, 6, 5, 4, 4,\dotsc$\\
  $36 $&$ \mathbf{15}, 5, 4, 4, 3,\dotsc$\\
  $38 $&$ \mathbf{{21 \over 2}, {21 \over 2}}, 5, 4, 4,\dotsc$\\
  $40 $&$ \mathbf{9}, 5, 4, 4, 3,\dotsc$\\
  $42 $&$ \mathbf{11}, 5, 5, 4, 4,\dotsc$\\
  $44 $&$ \mathbf{34}, 5, 4, 4, 3,\dotsc$\\
  $46 $&$ \mathbf{36}, 5, 5, 4, 4,\dotsc$\\
  $50 $&$ \mathbf{14}, 6, 5, 4, 4,\dotsc$\\
  \hline
\end{tabular}
\end{center}
\label{table:b4-zero-location}
\end{table}

Continuing then with the weight $k=38$, it was a ghost zero for $a_3$ with multiplicity one, multiplicity two for $a_4$ and one can check it should have multiplicity one for $a_5$ (see Table \ref{table:weight62-zeros})

With these computations in mind, we cataloged the relative location of the zeros of $a_5,a_6,\dotsc$ to the ghost zeros we were predicting. Seeing the data, and writing down the multiplicity $k$-by-$k$ we saw what became the multiplicity pattern:\ for each $k$, the first and last time appear of $k$ as a ghost zero it has multiplicity one, the second and second to last time it has multiplicity two, etc. To emphasize this, in Table \ref{table:weight62-zeros} below we give the relative positions of the zeros of each $a_i$ to the weights $w_{38}$ and $w_{62}$, with the multiplicity pattern emphasized through the use of bolding. 
\begin{table}[htpp]
\caption{Relative location of zeros of $a_1(w),\dotsc,a_{10}(w)$ for $w_{38}$ and $w_{62}$.  (Bold indicates the witnesses to $k$ as a ``ghost zero''.)}
\begin{center}
\begin{tabular}{|c|l|l|}
\hline
$i$ & $v_2(w_0-w_{38})$ where $a_i(w_0) = 0$ & $v_2(w_0-w_{62})$ where $a_i(w_0) = 0$ \\
\hline
 $1$ & $5$ &$ 6$\\
 $2$ & $6, 4, 3$ & $ 5, 4, 3$\\
 $3$ & $\mathbf{31}, 5, 4, 4, 3,\dotsc $ & $ 7, 5, 4, 4, 3,\dotsc$\\
 $4$ & $\mathbf{{21\over 2}, {21\over 2}}, 5, 4, 4,\dotsc$ & $ 6, 5, 5, 4, 4,\dotsc$\\
 $5$ & $\mathbf{22}, 6, 5, 5, 5,\dotsc$ & $ \mathbf{30}, 6, 6, 5, 5,\dotsc$\\
 $6$ & $7, 6, 6, 5, 5,\dotsc$ & $ \mathbf{14, 14}, 6, 5, 5,\dotsc$\\
 $7$ & $7, 7, 6, 6, 5,\dotsc$ & $\mathbf{29, {23\over 2}, {23\over 2}}, 6, 5,\dotsc$\\
 $8$ & $7, 7, 7, 6, 6,\dotsc$ & $ \mathbf{14, 14}, 7, 6, 6,\dotsc$\\
 $9$ & $8, 7, 7, 6, 6,\dotsc$ & $ \mathbf{30}, 7, 7, 6, 6,\dotsc$\\
 $10$ & $8, 8, 7, 6, 6,\dotsc$ & $ 7, 7, 7, 6, 6,\dotsc$\\
 \hline
\end{tabular}
\end{center}
\label{table:weight62-zeros}
\end{table}
\end{appendix}

\newpage

\bibliography{ghost_bib}
\bibliographystyle{abbrv}

\end{document}